\documentclass[10.5pt]{amsart}

\usepackage{subfiles}
\usepackage{euscript}
\usepackage{tabu}
\usepackage[margin=1in]{geometry}
\usepackage[dvipsnames]{xcolor} 	
\usepackage{verbatim}

\usepackage{graphicx}			
\usepackage{amssymb}
\usepackage{mathrsfs}
\usepackage{amsthm}
\usepackage{amsmath}
\usepackage{stmaryrd}

\usepackage{tikz}
\usepackage{tikz-cd}
\usetikzlibrary{arrows,arrows.meta}
\usetikzlibrary{arrows.meta,        
                decorations.pathmorphing,
                matrix}%
\tikzset{
  commutative diagrams/.cd, 
  arrow style=tikz, 
  diagrams={>=cm to}
}

\usepackage{accents}
\usepackage{upgreek}
\usepackage{enumerate}
\usepackage{bm}
\usepackage{mathtools}
\usepackage[all]{xy}
\usepackage{caption}

\usepackage{url}
\usepackage{float}
\usepackage{todonotes}
\usepackage{bbm}
\usepackage{colonequals}
\usepackage{longtable}

\usepackage[full]{textcomp}
\usepackage[cal=cm]{mathalfa}
\usepackage{xparse}
\usepackage{comment}
\usepackage{cite}

\usepackage[normalem]{ulem}

\theoremstyle{definition}
\newtheorem{innercustomthm}{Theorem}
\newenvironment{customthm}[1]
  {\renewcommand\theinnercustomthm{#1}\innercustomthm}
  {\endinnercustomthm}
  
\theoremstyle{definition}
\newtheorem{innercustomcor}{Corollary}
\newenvironment{customcor}[1]
  {\renewcommand\theinnercustomcor{#1}\innercustomcor}
  {\endinnercustomcor}
  
\theoremstyle{definition}

\makeatletter
\def\@tocline#1#2#3#4#5#6#7{\relax
  \ifnum #1>\c@tocdepth 
  \else
    \par \addpenalty\@secpenalty\addvspace{#2}%
    \begingroup \hyphenpenalty\@M
    \@ifempty{#4}{%
      \@tempdima\csname r@tocindent\number#1\endcsname\relax
    }{%
      \@tempdima#4\relax
    }%
    \parindent\z@ \leftskip#3\relax \advance\leftskip\@tempdima\relax
    \rightskip\@pnumwidth plus4em \parfillskip-\@pnumwidth
    #5\leavevmode\hskip-\@tempdima
      \ifcase #1
       \or\or \hskip 1em \or \hskip 2em \else \hskip 3em \fi%
      #6\nobreak\relax
    \dotfill\hbox to\@pnumwidth{\@tocpagenum{#7}}\par
    \nobreak
    \endgroup
  \fi}
\makeatother

\makeatletter
\DeclareRobustCommand{\cev}[1]{%
  \mathpalette\do@cev{#1}%
}
\newcommand{\do@cev}[2]{%
  \fix@cev{#1}{+}%
  \reflectbox{$\m@th#1\vec{\reflectbox{$\fix@cev{#1}{-}\m@th#1#2\fix@cev{#1}{+}$}}$}%
  \fix@cev{#1}{-}%
}
\newcommand{\fix@cev}[2]{%
  \ifx#1\displaystyle
    \mkern#23mu
  \else
    \ifx#1\textstyle
      \mkern#23mu
    \else
      \ifx#1\scriptstyle
        \mkern#22mu
      \else
        \mkern#22mu
      \fi
    \fi
  \fi
}

\makeatother

\makeatletter
\let\@wraptoccontribs\wraptoccontribs
\makeatother

\usetikzlibrary{calc}
\usetikzlibrary{fadings}
\usetikzlibrary{decorations.pathmorphing}
\usetikzlibrary{decorations.pathreplacing}

\newcounter{marginnote}
\DeclareMathAlphabet{\mathpzc}{OT1}{pzc}{m}{it}

\usepackage[backref=page]{hyperref}
\hypersetup{
  colorlinks   = true,          
  urlcolor     = blue,          
  linkcolor    = blue,          
  citecolor   = blue             
}
\theoremstyle{definition}
\newtheorem{theorem}{Theorem}[section]

\newtheorem{corollary}[theorem]{Corollary}
\newtheorem{lemma}[theorem]{Lemma}
\newtheorem{proposition}[theorem]{Proposition}
\newtheorem{remark}[theorem]{Remark}

\newtheorem*{runningexample*}{Running example}

\newtheorem*{aside*}{Aside}

\newtheorem{definition}[theorem]{Definition}
\newtheorem{example}[theorem]{Example}

\newtheorem{notation}[theorem]{Notation}
\newtheorem{proposition-definition}[theorem]{Proposition-Definition}

\DeclareMathOperator{\Pic}{Pic}

\newcommand{\Exp}{\operatorname{Exp}}

\newcommand{\Gm}{\mathbb{G}_{\operatorname{m}}}

\newcommand{\ol}[1]{\overline{#1}}

\newcommand{\bcd}{\begin{center}\begin{tikzcd}}
\newcommand{\ecd}{\end{tikzcd}\end{center}}

\newcommand{\G}{\mathbb{G}}
\newcommand{\PP}{\mathbb{P}}

\newcommand{\A}{\mathbb{A}}

\newcommand{\Mcal}{\mathcal{M}}

\newcommand{\Pcal}{\mathcal{P}}

\newcommand{\Mbar}{\ol{\Mcal}}

\newcommand{\M}{\mathcal{M}}

\newcommand{\calE}{\mathcal{E}}

\newcommand{\Spec}{\operatorname{Spec}}

\newcommand{\Hilb}{\mathrm{Hilb}}
\newcommand{\Sym}{\mathrm{Sym}}

\newcommand{\ra}{\rightarrow}
\newcommand{\CH}{\mathrm{CH}}
\newcommand{\CHs}{\mathrm{CH}^\star}

\newcommand{\cZ}{\mathcal{Z}}
\newcommand{\cE}{\mathcal{E}}
\newcommand{\cP}{\mathcal{P}}
\newcommand{\cX}{\mathcal{X}}
\newcommand{\cO}{\mathcal{O}}
\newcommand{\cY}{\mathcal{Y}}
\newcommand{\cN}{\mathcal{N}}
\newcommand{\cL}{\mathcal{L}}
\newcommand{\cA}{\mathcal{A}}
\newcommand{\cB}{\mathcal{B}}
\newcommand{\bi}{\boldsymbol{i}}
\newcommand{\bZ}{\mathbb{Z}}

\newcommand{\cZri}{\cZ}
\newcommand{\cEri}{\cE}
\newcommand{\piri}{\pi}
\newcommand{\Himor}{\cY}
\newcommand{\Hi}{\cX}

\NewDocumentCommand{\compatibilitydatum}{m m m m m m O{} O{} O{}}{
\begin{equation*} \begin{tikzcd}[ampersand replacement=\&]
  \: \arrow{r} \& {#1} \arrow{r} \arrow{d}{#7} \& {#2} \arrow{r} \arrow{d}{#8} \& {#3} \arrow{r}{[1]} \arrow{d}{#9} \& \: \\
  \: \arrow{r} \& {#4} \arrow{r} \& {#5} \arrow{r} \& {#6} \arrow{r} \& \:
\end{tikzcd} \end{equation*}}

\NewDocumentCommand{\commutingsquare}{m m m m o O{} O{} O{} O{}}{
\begin{equation}\begin{tikzcd}[ampersand replacement=\&] \label{#5}
  #1 \arrow{r}{#6} \arrow{d}{#7} \& #2 \arrow{d}{#8} \\
  #3 \arrow{r}{#9} \& #4
\end{tikzcd}\IfValueTF{#5}{\label{#5}}{} \end{equation}}

\NewDocumentCommand{\Cartesiansquare}{m m m m O{} O{} O{} O{}}{
\begin{equation*}\begin{tikzcd}[ampersand replacement=\&]
  #1 \arrow{r}{#5} \arrow{d}{#6} \arrow[dr, phantom, "\square"] \& #2 \arrow{d}{#7} \\
  #3 \arrow{r}{#8} \& #4
\end{tikzcd} \end{equation*}}

\NewDocumentCommand{\Cartesiansquarelabel}{m m m m m O{} O{} O{} O{}}{
\begin{tikzcd}[ampersand replacement=\&]
  #1 \arrow{r}{#6} \arrow{d}{#7} \arrow[dr, phantom, "\square"] \& #2 \arrow{d}{#8} \\
  #3 \arrow{r}{#9} \& #4
\end{tikzcd}\IfValueTF{#5}{\label{#5}}{}
}

\NewDocumentCommand{\triangleofspaces}{m m m O{} O{} O{}}{
\begin{tikzcd} [ampersand replacement=\&]
#1 \arrow{r}{#4} \arrow[bend right]{rr}{#5} \& #2 \arrow{r}{#6} \& #3
\end{tikzcd}}

\definecolor{orange2}{RGB}{255,180,0}

\begin{document}
\title{Logarithmic Hilbert schemes of curves as weighted blow-ups and their integral Chow rings}

\author[V. Arena]{Veronica Arena}\address{Department of Pure Mathematics and Mathematical Statistics, University of Cambridge, Cambridge, CB3 0WA}\email{\url{va365@cam.ac.uk}}
\author[T. Song]{Terry Dekun Song}\address{Department of Pure Mathematics and Mathematical Statistics, University of Cambridge, Cambridge, CB3 0WA}\email{\url{ds2016@cam.ac.uk}}
\contrib[with an appendix by]{Dhruv Ranganathan}

\begin{abstract}
    The logarithmic Hilbert scheme of a logarithmic curve parametrizes subschemes on the expanded degenerations of the curve that are transverse to the boundary. We prove that the logarithmic Hilbert scheme of points on a smooth pointed curve is an iterated weighted blow-up of the symmetric product of the underlying curve. In doing so, we explicitly identify the blow-up centers, weights, and give them modular interpretations. As applications, we calculate their integral Chow rings in terms of those of the symmetric products. Key ingredients in our work include two recent results: the integral Chow ring formula of weighted blow-ups and a weighted analogue of Castelnuovo's criterion for blow-downs. We recover the folklore result the logarithmic Hilbert scheme of toric $\PP^1$ is a toric stack, and the Appendix by Dhruv Ranganathan outlines a complementary approach using Chow quotients.
\end{abstract}
\maketitle
\setcounter{tocdepth}{1}
\tableofcontents
\section{Introduction}
Let $C$ be a smooth, projective curve over a field $k$ of characteristic zero. Distinct closed points $p_1,\dots,p_\ell\in C$ give $C$ a divisorial logarithmic structure as a pair $(C| \sum_{i=1}^\ell p_i) = (C|D).$ The \emph{logarithmic Hilbert scheme} of $n$ points on $(C|D),$ denoted as $\Hilb^n(C|D),$ parametrizes pairs $(\tilde{C}, Z)$ where 
\begin{enumerate}
\item $\tilde{C}$ is a semistable model of $C$ obtained from iterated deformations to the normal cone along $D,$
\item $Z\subset \tilde{C}$ is a length-$n$ closed subscheme that is logarithmically flat and has finite automorphism {over $C$}.
\end{enumerate} The precise definitions are recalled in Section \ref{sec:background}. The space $\Hilb^n(C|D)$ is representable by a Deligne--Mumford stack with logarithmically smooth logarithmic structure {and non-trivial stacky structure}. It is also smooth as an algebraic stack with projective coarse moduli space that compactifies $\Hilb^n(C\setminus D).$ {Broader interests in general logarithmic Hilbert schemes come from studying enumerative geometry of curves and sheaves via degenerations \cite{lw15, MR20, logquot, MR25}, and $\Hilb^n(C|D)$ serves as a local model of some of the more general constructions.}

This work concretely describes the logarithmic Hilbert scheme of points on a curve via the geometry of \textit{weighted blow-ups} (c.f. \cite{quek-rydh-weighted-blow-up}). Given $(\tilde{C}, Z)\in \Hilb^n(C|D)$, taking the scheme-theoretic image along the collapsing map $\tilde{C}\to C$ induces a map $\Hilb^n(C|D)\to \Hilb^n(C)$ from the logarithmic moduli space to the classical one. Our main geometric result is as follows.

\begin{customthm}{\ref{thm:main}}
  The map $\Hilb^n(C|D)\to \Hilb^n(C) = \Sym^n(C)$ factors through a sequence of $\ell \cdot n$ weighted blow-ups $$\Hilb^n(C|D)\to \cdots  \to \Hilb^n(C|D)_{\leq \bi}  \to \cdots \to \Hilb^n(C)$$ where $\bi$ belongs to a totally ordered subset of all  $\bi$ satisfying $\mathbf{0}\leq \bi\leq \mathbf{n} = (n,\dots,n)$ entry-wise. Each $\Hilb^n(C|D)_{\leq \bi}$ is {a moduli space} of subschemes on expanded degenerations of $C$ along $D$ satisfying intermediate stability conditions.
\end{customthm}

The moduli spaces $\Hilb^n(C|D)_{\leq \bi}$ are constructed in Section \ref{sec:interm}. The blow-up centers and weights are explicitly described in Section \ref{sec:wbu}. The edge cases $\Hilb^n(C|D)_{\leq \mathbf{0}}$ and $\Hilb^n(C|D)_{\leq \mathbf{n}}$ recover $\Hilb^n(C|D)$ and $\Hilb^n(C)$ respectively. Interpolating between them are $\Hilb^n(C|D)_{\leq \bi}$ for $\mathbf{0}\lneq \bi\lneq \mathbf{n}$. Given $\bi_1\leq \bi_2,$ the map \[\Hilb^n(C|D)_{\leq \bi_2}\to \Hilb^n(C|D)_{\leq \bi_1}\] can be seen as improving the transversality of subschemes parametrized by $\Hilb^n(C|D)_{\leq \bi_1}$ at the cost of introducing more expanded degenerations in $\Hilb^n(C|D)_{\leq \bi_2}.$ 

Our proof of the theorem reduces to the case $D = p,$ which we now spell out. In this case, each non-trivial semistable model $\tilde{C}\to C$ is $C$ attached with a chain of rational curves bubbling off from $p\in C$. Let $\tilde{p}\in \tilde{C}$ be the proper transform of $p\in C.$ It is the unique smooth point with non-trivial logarithmic structure.


\begin{customthm}{\ref{thm:D=p}}
  For $0\leq i\leq n,$ define $\cZ_i\subset \mathrm{Hilb}^n(C|p)_{\leq i}$ as the closed substack parametrizing $(\tilde{C}, Z)$ where exactly $i$ points are supported on $\tilde{p}.$ Then $ \mathrm{Hilb}^n(C|p)_{\leq i-1}\to \mathrm{Hilb}^n(C|p)_{\leq i}$ is the weighted blow-up with center $\cZ_i$ and weights $(1,\dots,i).$
\end{customthm}

To prove this in Section \ref{sec:wbu}, we identify a divisor $\mathcal{E}_{i-1}\subset \Hilb^n(C|p)_{\leq i-1}$ and prove that it satisfies the smooth weighted blow-down criterion of \cite{arena2024criterion}.

\subsection{Toric geometry on $\mathbb{P}^1$}
We note that the geometry of the logarithmic Hilbert stacks is particularly concrete when $C = \PP^1$ and $D\leq 0+\infty.$ In this case,

\begin{customcor}{\ref{cor:P1}}
     The logarithmic Hilbert stacks $\Hilb^n(\PP^1|0)$ and $\Hilb^n(\PP^1|0+\infty)$ are toric stacks.
\end{customcor}

Indeed, the above logarithmic Hilbert stacks are weighted blow-ups of $\PP^n\cong \Sym^n(\PP^1)$ along torus-invariant strata. Tracking the weights in the blow-ups, we identify the toric fan of $\Hilb^n(\PP^1|0)$ as the successive star subdivisions of the fan of $\PP^n$ along the rays $(1,2,\dots,n)$, $(0,1,2,\dots, n-1)$,$\dots$, $(0,0,\dots,1,2).$ Each ray represents the weighted blow-up along the toric stratum of the zero locus of some symmetric functions on $\PP^1$, and the weights of the blow-up coincide with the degrees of the symmetric functions.

In particular, we describe the explicit construction of $\Hilb^2(\PP^1|0)$ as weighted blow-up in Section \ref{sec:2P10}. This section gives a self-contained and accessible illustration of the geometric picture. 

\subsection{Integral Chow ring calculation} After the geometric description of $\Hilb^n(C|D)$ given by Theorem \ref{thm:main}, we leverage the weighted Keel's formula in \cite{Arena2025}, which gives a presentation of the integral Chow rings of the moduli spaces $\CHs(\Hilb^n(C|D)_{\le \bf i})$ as generated by $\mathrm{CH}^\star(\mathrm{Sym}^n(C))$ and the exceptional divisors $[\cE_{\bi}]$ introduced at each weighted blow-up. 

To state our result, we set up the following notation. For $p_r\in D,$ define line bundle $\cL_{p_r}$ on $\Sym^n(C)$ as the normal bundle along the embedding $s_{n, n+1}^{(p_r)}:\Sym^n(C)\hookrightarrow \Sym^{n+1}(C)$ given by $Z\mapsto Z+p_r.$ The line bundles $\cL_{p_r}$ on $\Sym^n(C)$ are preserved under pullback maps $(s_{n,n+1}^{(p)})^*$ for any $p\in C.$

\begin{customthm}{\ref{thm:chowringD}}
    Let $\epsilon_{j}^{(r)}\in \CH^1(\Hilb^{n}(C|D)_{\leq \bi})$ be the exceptional divisor corresponding to $j$ points supported on the last bubble along $p_r.$ The Chow ring $\mathrm{CH}^\star(\Hilb^n(C|D)_{\leq \bi})$ has a presentation $$\mathrm{CH}^\star(\Hilb^n(C|D)_{\leq \bi}) = \frac{\mathrm{CH}^\star(\mathrm{Sym}^n(C))[\epsilon^{(r)}_{j}\mid i_r+1\leq j\leq n, r = 1,\dots, \ell]}{\left\{\begin{array}{c}
   \text{for each }j = n, \dots, i_r+1:\  Q_{n,j}^{(r)}(\epsilon^{(r)}_{n},\dots,\epsilon^{(r)}_{j}); \vspace{5pt}\\  \ker({(s^{(p_r)}_{(n-j), n})}^*) \cdot \epsilon^{(r)}_{j}; \vspace{5pt}\\  \text{for each } k = n-j, \dots, 1: \ Q_{n-j,k}^{(r)}(\epsilon^{(r)}_{n},\dots, \epsilon^{(r)}_{j+k})\cdot\epsilon^{(r)}_{n-j} 
  \end{array}\right\}_{r = 1,\dots, \ell}}$$ where $Q_{m,h}$ are multivariable polynomials (each corresponding to a top Chern class) with coefficients in $\langle c_1(\mathcal{L}_{p_r}): r = 1,\dots, \ell\rangle\subset\CH^\star(\Sym^n(C))$ defined as: \begin{enumerate}
      \item $Q_{m,h}(t_m, \dots, t_{h+1}, t):=\prod_{k=1}^h (kt + c_1(\mathcal{L})-\sum_{j=1}^{m-h}(k+j)t_{j+h})$ when $m>h>0,$
      \item $Q_{m,m}(t):=\prod_{k=1}^h (kt + c_1(\mathcal{L}))$ when $m = h,$
      \item formally $Q_{0, k} = 0$ for all $k.$
  \end{enumerate}
\end{customthm}

\subsection{Euler characteristics} In Section \ref{sec:chi}, we use the stratification of $\Hilb^n(C|D)$ to compute the generating function of the classes $[\Hilb^n(C|D)]\in K_0(\mathsf{Var}),$ the Grothendieck ring of varieties.
\begin{customthm}{\ref{thm:chi}}
    We have $$\sum_{n\geq 0}[\Hilb^n(C|D)] t^n = Z_C(t)\left(\frac{(1-\mathbb{L}t)(1-t)}{1-(\mathbb{L}+1)t}\right)^\ell,$$ where $Z_C(t)= \sum_{m\geq 0} [\mathrm{Sym}^m(C)] t^m\in K_0(\mathsf{Var})[\![t]\!]$ is the motivic Zeta function of $C.$
\end{customthm}

The calculation specializes to, for instance, the generating function of Euler characteristics $$\sum_{n\geq 0}\chi(\Hilb^n(C|D)) t^n = (1-t)^{2g-2} \left(\frac{(1-t)^2}{1-2t}\right)^{\ell}.$$

\subsection{Perspectives}
{While logarithmic moduli spaces are primarily motivated by applications in enumerative geometry, it is of independent interest to relate their geometry to the well-studied non-logarithmic moduli spaces. Our work gives a concrete first example in this direction.}

\label{subsec:modexp}
Let $\mathscr{C}\to \Exp(C|D)$ be the universal curve over the stack $\Exp(C|D)$ of expanded degenerations, and let $\Hilb^n(\mathscr{C})$ be the relative Hilbert scheme of length $n,$ so that the morphism $\Hilb^n(\mathscr{C})\to \Exp(C|D)$ is representable. While being universally closed, the stack $\Hilb^n(\mathscr{C})$ is neither quasi-compact nor separated and has positive dimensional stabilizers. 

The intermediate moduli stacks $\Hilb^n(C|D)_{\leq \bi}$ are open substacks of $\Hilb^n(\mathscr{C})$ that are smooth, proper, and contain $\Hilb^n(C\setminus D)$ as a dense open substack. {In other words, they are candidates to produce well-behaved moduli spaces from the stack $\Hilb^n(\mathscr{C}).$ From this point of view, the iterated weighted blow-up descriptions in Theorem \ref{thm:main} set up a VGIT-style wall-crossing between logarithmic and classical Hilbert scheme of points on a smooth marked curve, in which the $\mathbb{G}_m$-quotients arise from action by the rubber torus.} Instances of weighted blow-ups as wall-crossing morphisms include Shah's work on degree-2 K3 surfaces \cite{Shah} and its application to intersection theory \cite{COP}, K-moduli of plane curves \cite{ADL}, and moduli of stable genus one curves \cite[§7]{arena2024criterion}.

More generally, given a moduli stack $\mathscr{M}$ over (fine and saturated) logarithmic schemes, there often exists a piecewise linear space $\mathscr{S}$--this is a stack over fs log schemes that admit log modifications by Artin fans \cite{logquot}--and a morphism $\mathscr{M}\to \mathscr{S}$ that is representable by an algebaric stack. Producing proper open substacks of $\mathscr{M}$ and studying their wall-crossing morphisms is a systematic way to relate logarithmic moduli spaces and their non-logarithmic parallels.




\subsection{Related work}
{The general moduli problem of logarithmic Hilbert schemes is motivated by approaching enumerative geometry via degenerations: they are constructed for smooth pairs by Li--Wu \cite{lw15}, for curves in simple normal crossings pair three-folds by Maulik--Ranganathan \cite{MR20} and in full generality by Kennedy-Hunt \cite{logquot}. Examples and their concrete geometric descriptions point to connections to the toric geometry of secondary polytope by work of Kennedy-Hunt on linear systems of curves in toric surfaces \cite{kh21} and geometric invariant theory by work of Tschanz and Shafi--Tschanz on logarithmic Hilbert schemes of points on degenerations of K3 surfaces \cite{tschanz2023, ST25}.} We also acknowledge the forthcoming work of Suraj Dash, which gives a combinatorial classification of GIT quotients arising from the relative Hilbert schemes of points over the stack of expansions in a more general setting.

\subsection{Future directions}
As logarithmic Hilbert schemes are motivated by the study of moduli spaces via degeneration, it is natural to investigate logarithmic Hilbert and quot schemes of nodal curves {as well as the relative logarithmic moduli spaces of the universal curve over $\Mbar_{g,n}$}. {Analogous wall-crossing morphisms relating them to the ordinary moduli spaces can potentially lead to comparison formulas between enumerative invariants coming from the logarithmic and non-logarithmic moduli spaces.}


The weighted blow-up picture explored in this work seems to be applicable to general logarithmic symmetric products and has application to the theory of logarithmic 0-cycles. In the surfaces case, combining this approach with \cite[§4]{MR25} should produce a complete description of the cohomology and Chow ring of the logarithmic Hilbert schemes of points on surfaces. 

{Viewing smooth points as divisors on curves, it will be worthwhile to study the moduli of logarithmic hypersurfaces from the perspective of wall-crossing. For instance, it will be interesting to produce intermediate moduli spaces that interpolate between the logarithmic linear system of curves on toric surfaces studied by Kennedy-Hunt \cite{kh21} and the ordinary linear system.}

\subsection{Acknowledgments}
We are grateful to Dhruv Ranganathan for suggesting the project, helping us with many enlightening discussions and for the Appendix. We also thank Dan Abramovich, María Abad Aldonza, Samir Canning, Francesca Carocci, Suraj Dash, Andrea Di Lorenzo, Giovanni Inchiostro, Patrick Kennedy-Hunt, Siddarth Mathur, Denis Nesterov, Michele Pernice, and Thibault Poiret for helpful conversations.

Terry Song is supported by a Cambridge Trust international scholarship. Veronica Arena is supported by EPSRC Horizon Europe Guarantee EP/Y037162/1.

\subsection{Notes to the reader} In the following, the term `weights' is overloaded as it is both used to denote the weights of $\mathbb{G}_m$-action and the weights of Hassett spaces. Despite this, we expect no serious confusion to arise.

\section{Background on logarithmic Hilbert schemes}\label{sec:background}

In this section, we review basic facts and properties of the logarithmic Hilbert scheme that will be relevant for our constructions in the next section.

\subsection{Expanded degenerations}

Logarithmic Hilbert schemes of a curve parametrize subschemes on \emph{expanded degenerations} of the curve. Given a curve $C$ and a divisor $D,$ an expanded degeneration, or \emph{expansion} for short, is a semistable model of $C$ that destabilizes along the direction of $D.$ After defining it on a single curve, we proceed to recall facts on the moduli stack of expansions that we denote as $\mathrm{Exp}(C|D).$

\begin{definition}
    Let $C$ be a curve, and let $p\in C$ be a smooth point. The one-step expansion of $C$ along $p$ is the special fibre $\tilde{C}$ of the deformation to the normal cone construction of $p\in C.$ There is a proper contraction map $\pi: \tilde{C}\to C$ that is an isomorphism away from $p\in C,$ and $\pi^{-1}(p) \cong \mathbb{P}(N_{p/C}\oplus \mathbf{1})$ is called the \emph{bubble} of the one-step expansion.
\end{definition}

\begin{definition}\label{defn:singleexp}
    Let $C$ be a curve, and let $p_1,\dots,p_\ell\in C$ be distinct closed points with $D = \sum_{i=1}^\ell p_i.$ An \textbf{expansion} of $C$ along $D$ is a proper map $\tilde{C}\to C$ that factors through a sequence of one-step expansions along any $p_i\in C$ or their proper transforms\footnote{It refers to taking the proper transforms along the iterated blow-ups that construct the expansion stack $\mathrm{Exp}(C|p):$ See \cite[§1]{Li2001} or \cite[§1]{MT16}.} in previous one-step expansions. We denote $\tilde{p}_j\in \tilde{C}$ as the proper transform of $p_j\in C$ and call it the end point of $\tilde{C}$ along $p_j.$
       A \textit{bubble} in $\tilde{C}$ is any component of $\tilde{C}$ that is contracted along the map $\tilde{C}\to C.$ A \textit{chain of bubbles} is a union of bubbles that is connected. An \textit{automorphism} of an expansion $\tilde{C}\to C$ is an automorphism $\sigma\in \mathrm{Aut}(\tilde{C},\tilde{p}_1,\dots,\tilde{p}_\ell)$ that fits in the commutative diagram \[\begin{tikzcd}
	{(\tilde{C},\tilde{p}_1,\dots,\tilde{p}_\ell)} && {(\tilde{C},\tilde{p}_1,\dots,\tilde{p}_\ell)} \\
	& C
	\arrow["\sigma", from=1-1, to=1-3]
	\arrow[from=1-1, to=2-2]
	\arrow[from=1-3, to=2-2]
\end{tikzcd}\]
\end{definition}

\begin{remark}\label{rem:autgroups}
    Let $\pi: \tilde{C}\to C$ be the contraction map, and let $R^{(i)}\subset \tilde{C}$ be $\pi^{-1}(\tilde{p}_i),$ so that $R^{(i)}$ is a chain of rational curves. From Definition \ref{defn:singleexp}, we see that the automorphism group of the point $(\tilde{C},\tilde{p}_1,\dots,\tilde{p}_\ell)$ is a product of algebraic tori $\prod_{i=1}^\ell \mathbb{G}_m^{n_i},$ where $n_i$ is the number of rational components of each $R^{(i)}.$ On the other hand, let $C$ have arithmetic genus $g,$ then $(\tilde{C},\tilde{p}_1,\dots, \tilde{p}_\ell)$ defines a point in $\mathfrak{M}_{g,\ell},$ the stack of genus $g$ nodal curves with $\ell$ marked points, and the automorphism group of $(\tilde{C},\tilde{p}_1,\dots,\tilde{p}_\ell)$ as an expansion agrees with its automorphism group in the stack $\mathfrak{M}_{g,\ell}.$
\end{remark}

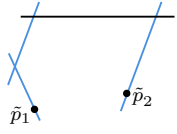
\begin{figure}[htp]
\resizebox{.15\textwidth}{!}{

\tikzset{every picture/.style={line width=0.75pt}} 

\begin{tikzpicture}[x=0.75pt,y=0.75pt,yscale=-1,xscale=1]

\draw [color={rgb, 255:red, 74; green, 144; blue, 226 }  ,draw opacity=1 ]   (93.04,15.97) -- (71.3,74.28) ;
\draw [color={rgb, 255:red, 74; green, 144; blue, 226 }  ,draw opacity=1 ]   (27.73,16) -- (11.17,60.33) ;
\draw [color={rgb, 255:red, 0; green, 0; blue, 0 }  ,draw opacity=1 ]   (17.94,23.84) -- (99.98,23.84) ;
\draw [color={rgb, 255:red, 74; green, 144; blue, 226 }  ,draw opacity=1 ]   (11.56,43.42) -- (28.17,79.33) ;
\draw  [fill={rgb, 255:red, 0; green, 0; blue, 0 }  ,fill opacity=1 ] (75.96,64.84) .. controls (75.96,64.02) and (75.3,63.36) .. (74.48,63.36) .. controls (73.66,63.36) and (73,64.02) .. (73,64.84) .. controls (73,65.66) and (73.66,66.32) .. (74.48,66.32) .. controls (75.3,66.32) and (75.96,65.66) .. (75.96,64.84) -- cycle ;
\draw  [fill={rgb, 255:red, 0; green, 0; blue, 0 }  ,fill opacity=1 ] (26.78,73.33) .. controls (26.78,72.51) and (26.12,71.85) .. (25.3,71.85) .. controls (24.48,71.85) and (23.82,72.51) .. (23.82,73.33) .. controls (23.82,74.15) and (24.48,74.82) .. (25.3,74.82) .. controls (26.12,74.82) and (26.78,74.15) .. (26.78,73.33) -- cycle ;

\draw (11,70.4) node [anchor=north west][inner sep=0.75pt]  [font=\footnotesize]  {$\tilde{p}_{1}$};
\draw (76,60) node [anchor=north west][inner sep=0.75pt]  [font=\footnotesize]  {${\tilde{p}}_{2}$};

\end{tikzpicture}

}
\caption{An expansion of $(C|p_1+p_2)$ with proper transforms $\tilde{p}_1$ and $\tilde{p}_2$.}
\label{fig:exp}
\end{figure} 

\begin{remark}
    A bubble in an expansion is always isomorphic to $\mathbb{P}(T_{\tilde{p}_i'}\tilde{C}'\oplus \mathbf{1})\cong \mathbb{P}^1,$ where $\tilde{C}'$ is an intermediate expansion and $\tilde{p}_i'\in C'$ is the proper transform of some marked point $p_i,$ hence a smooth point on $\tilde{C}'.$ The bubble is naturally endowed with a toric structure with $0 = \mathbb{P}(T_{\tilde{p}_i'}\tilde{C}'\oplus 0)$ and $\infty = \mathbb{P}(0\oplus \mathbf{1}).$\end{remark}

\subsection{The stack Exp}\label{subsec:stackexp}

We list the basic properties and facts about the stack of expanded degenerations $$\mathrm{Exp}(C|p) = \mathrm{Exp}(\mathbb{A}^1/\mathbb{G}_m)$$ that we will use for later. It is first introduced by J. Li in \cite[§1]{Li2001}. We also refer the reader to \cite[§2.5]{GraberVakil}, \cite[§2-3]{ACFW}, \cite[§1]{MT16}, and \cite[§3.3]{Oesinghaus18}.

\begin{definition}
    Let $\mathsf{Exp}^\bullet(\mathbb{R}_{\geq 0})$ be the following diagram of polyhedral cones $$\mathrm{pt} = \mathbb{R}_{\geq 0}^0\to\mathbb{R}_{\geq 0}^1 \rightrightarrows  \mathbb{R}_{\geq 0}^2 \substack{\rightarrow\\[-1em] \rightarrow \\[-1em] \rightarrow} \begin{tikzcd}{\mathbb{R}_{\geq 0}^3} & {\cdots} & {\mathbb{R}_{\geq 0}^{n}} & & {\mathbb{R}_{\geq 0}^{n+1}} & &\cdots
	\arrow["4"{description}, from=1-1, to=1-2]
    \arrow["n"{description}, from=1-2, to=1-3]
    \arrow["n+1"{description}, from=1-3, to=1-5]
    \arrow["n+2"{description}, from=1-5, to=1-7]
\end{tikzcd}$$ where each $\mathbb{R}_{\geq 0}^n$ has $n+1$ maps into $\mathbb{R}_{\geq 0}^{n+1}$ given by mapping the $n$ coordinates to $\mathbb{R}_{\geq 0}^{n+1}$ via the $n+1$ ordered injections $\{1,\dots, n\}\hookrightarrow \{1,\dots, n+1\}$ and extending by zero.
\end{definition}

For instance, the two injections $\mathbb{R}_{\geq 0}^1\rightrightarrows \mathbb{R}_{\geq 0}^2$ are given by $x\mapsto (0,x)$ and $x\mapsto (x,0).$ We think of $(x_1,\dots, x_n)\in \mathbb{R}_{\geq 0}^n$ as parametrizing the edge lengths of $n$ ordered edges and think of $\mathsf{Exp}^\bullet(\mathbb{R}_{\geq 0})$ as the combinatorial moduli space of subdividing the ray $\mathbb{R}_{\geq 0}.$ The following figure indicates the case of $(x_1,x_2)\in \mathbb{R}_{\geq 0}^2.$ The meaning of the ray will be explained momentarily.

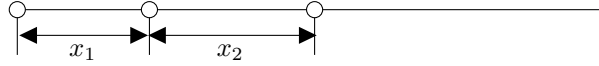
\[\begin{tikzpicture}[x=0.75pt,y=0.75pt,yscale=-1,xscale=1]

\draw   (23.13,27.04) .. controls (23.13,24.81) and (24.94,23) .. (27.17,23) .. controls (29.4,23) and (31.2,24.81) .. (31.2,27.04) .. controls (31.2,29.27) and (29.4,31.07) .. (27.17,31.07) .. controls (24.94,31.07) and (23.13,29.27) .. (23.13,27.04) -- cycle ;
\draw   (89.45,27.04) .. controls (89.45,24.81) and (91.25,23) .. (93.48,23) .. controls (95.71,23) and (97.52,24.81) .. (97.52,27.04) .. controls (97.52,29.27) and (95.71,31.07) .. (93.48,31.07) .. controls (91.25,31.07) and (89.45,29.27) .. (89.45,27.04) -- cycle ;
\draw   (172.34,27.04) .. controls (172.34,24.81) and (174.15,23) .. (176.38,23) .. controls (178.61,23) and (180.42,24.81) .. (180.42,27.04) .. controls (180.42,29.27) and (178.61,31.07) .. (176.38,31.07) .. controls (174.15,31.07) and (172.34,29.27) .. (172.34,27.04) -- cycle ;
\draw    (31.2,27.04) -- (89.45,27.04) ;
\draw    (97.52,27.04) -- (172.34,27.04) ;
\draw    (180.42,27.04) -- (323,27.04) ;
\draw    (31.14,39.93) -- (89.8,39.93) ;
\draw [shift={(92.8,39.93)}, rotate = 180] [fill={rgb, 255:red, 0; green, 0; blue, 0 }  ][line width=0.08]  [draw opacity=0] (8.93,-4.29) -- (0,0) -- (8.93,4.29) -- cycle    ;
\draw [shift={(28.14,39.93)}, rotate = 0] [fill={rgb, 255:red, 0; green, 0; blue, 0 }  ][line width=0.08]  [draw opacity=0] (8.93,-4.29) -- (0,0) -- (8.93,4.29) -- cycle    ;
\draw    (93.48,31.07) -- (93.48,49.74) ;
\draw    (27.17,31.07) -- (27.17,49.74) ;
\draw    (176.38,31.07) -- (176.38,49.74) ;
\draw    (95.8,39.93) -- (174.44,39.93) ;
\draw [shift={(177.44,39.93)}, rotate = 180] [fill={rgb, 255:red, 0; green, 0; blue, 0 }  ][line width=0.08]  [draw opacity=0] (8.93,-4.29) -- (0,0) -- (8.93,4.29) -- cycle    ;
\draw [shift={(92.8,39.93)}, rotate = 0] [fill={rgb, 255:red, 0; green, 0; blue, 0 }  ][line width=0.08]  [draw opacity=0] (8.93,-4.29) -- (0,0) -- (8.93,4.29) -- cycle    ;

\draw (52,43.4) node [anchor=north west][inner sep=0.75pt]    {$x_{1}$};
\draw (126,43.4) node [anchor=north west][inner sep=0.75pt]    {$x_{2}$};
\end{tikzpicture}\]

We recall that the cone $\mathbb{R}_{\geq 0}^n$ has the Artin fan $[\mathbb{A}^n/\mathbb{G}_m^n] = [\mathbb{A}^1/\mathbb{G}_m]^n$ via the language of \cite{ccuw}. Because the maps between $\mathbb{R}_{\geq 0}^n$ described above are maps of polyhedral cone complexes, they induce algebraic morphisms between their Artin fans.
\begin{definition}
    Let $\mathrm{Exp}^{\bullet}(\mathbb{A}^1/\mathbb{G}_m)$ be the diagram of Artin fans of the cones $$\mathrm{Spec}(k) \to \mathbb{A}^1/\mathbb{G}_m \rightrightarrows 
	\begin{tikzcd}{[\mathbb{A}^1/\mathbb{G}_m]^2} & {\cdots} & {[\mathbb{A}^1/\mathbb{G}_m]^n} & & {[\mathbb{A}^1/\mathbb{G}_m]^{n+1}} & &\cdots 
	\arrow["3"{description}, from=1-1, to=1-2]
    \arrow["n"{description}, from=1-2, to=1-3]
    \arrow["n+1"{description}, from=1-3, to=1-5]
    \arrow["n+2"{description}, from=1-5, to=1-7]
\end{tikzcd} $$
such that each map is the open immersion of Artin stacks corresponding to the maps of the polyhedral cones defined above.
\end{definition}

\begin{lemma}
    The stack $\mathrm{Exp}(C|p) = \mathrm{Exp}(\mathbb{A}^1/\mathbb{G}_m)$ is the colimit of the diagram $\mathrm{Exp}^{\bullet}(\mathbb{A}^1/\mathbb{G}_m)$.
\end{lemma} 

\begin{remark}
    The diagrams above give the combinatorial moduli space of subdivisions of $\mathbb{R}_{\geq 0},$ which we denote as $\mathsf{Exp}(\mathbb{R}_{\geq 0}),$ a well-defined cone complex structure as the colimit of the diagram $\mathsf{Exp}^\bullet(\mathbb{R}_{\geq 0})$. In general, choices are required to give moduli spaces of subdivisions of higher dimensional polyhedral cones structures of cone complexes. This corresponds to the fact that $\mathrm{Exp}(\mathbb{A}^1/\mathbb{G}_m)$ is a well-defined algebraic stack with logarithmic structure, while its higher dimensional analogues are not algebraic stacks in general.
\end{remark}

We relate the stack to the moduli problem of expanded degenerations. Let $x_i$ be a coordinate in the cone $\mathbb{R}_{\geq 0}^n.$ It corresponds to a factor of $\mathbb{A}^1/\mathbb{G}_m$ in the Artin fan $[\mathbb{A}^1/\mathbb{G}_m]^n,$ and the underlying topology of $\mathbb{A}^1/\mathbb{G}_m$ is given by the specialization $\mathrm{Spec}(k)\rightsquigarrow B\mathbb{G}_m.$ The dense point $\mathrm{Spec}(k)$ corresponds not performing any expansion, so the pair $(C,p)$ tropicalizes to:
\[\begin{tikzpicture}[x=0.75pt,y=0.75pt,yscale=-1,xscale=1]

\draw   (15.34,23.04) .. controls (15.34,20.81) and (17.15,19) .. (19.38,19) .. controls (21.61,19) and (23.42,20.81) .. (23.42,23.04) .. controls (23.42,25.27) and (21.61,27.07) .. (19.38,27.07) .. controls (17.15,27.07) and (15.34,25.27) .. (15.34,23.04) -- cycle ;
\draw    (23.42,23.04) -- (166,23.04) ;
\end{tikzpicture}\]
The closed point $B\mathbb{G}_m$ corresponds to a one-step expansion of the curve $C$ along $p$, and the stabilizer group of $\mathbb{G}_m$ refers to the automorphism group of the bubble relative to the two marked points. The tropicalization of the expanded pair $(\tilde{C},\tilde{p})$ is given as follows, in which the vertex next to the ray corresponds to the bubble:

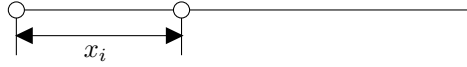
\[\begin{tikzpicture}[x=0.75pt,y=0.75pt,yscale=-1,xscale=1]

\draw   (13.45,13.04) .. controls (13.45,10.81) and (15.25,9) .. (17.48,9) .. controls (19.71,9) and (21.52,10.81) .. (21.52,13.04) .. controls (21.52,15.27) and (19.71,17.07) .. (17.48,17.07) .. controls (15.25,17.07) and (13.45,15.27) .. (13.45,13.04) -- cycle ;
\draw   (96.34,13.04) .. controls (96.34,10.81) and (98.15,9) .. (100.38,9) .. controls (102.61,9) and (104.42,10.81) .. (104.42,13.04) .. controls (104.42,15.27) and (102.61,17.07) .. (100.38,17.07) .. controls (98.15,17.07) and (96.34,15.27) .. (96.34,13.04) -- cycle ;
\draw    (21.52,13.04) -- (96.34,13.04) ;
\draw    (104.42,13.04) -- (247,13.04) ;
\draw    (17.48,17.07) -- (17.48,35.74) ;
\draw    (100.38,17.07) -- (100.38,35.74) ;
\draw    (19.8,25.93) -- (98.44,25.93) ;
\draw [shift={(101.44,25.93)}, rotate = 180] [fill={rgb, 255:red, 0; green, 0; blue, 0 }  ][line width=0.08]  [draw opacity=0] (8.93,-4.29) -- (0,0) -- (8.93,4.29) -- cycle    ;
\draw [shift={(16.8,25.93)}, rotate = 0] [fill={rgb, 255:red, 0; green, 0; blue, 0 }  ][line width=0.08]  [draw opacity=0] (8.93,-4.29) -- (0,0) -- (8.93,4.29) -- cycle    ;

\draw (50,29.4) node [anchor=north west][inner sep=0.75pt]    {$x_{i}$};

\end{tikzpicture}\]

The modular interpretation of the Artin fans $[\mathbb{A}^1/\mathbb{G}_m]^n$ now follows from taking product across all coordinates in $\mathbb{R}_{\geq 0}^n.$

\begin{definition}\label{defn:scrCtilp}
    Let $\mathscr{C}_p\to \mathrm{Exp}(C|p)$ be the universal curve over the expansion stack. From our discussion above, it has open substack $C\subset \mathscr{C}.$ 
    The \emph{endpoint} along $p$ defines a section $\tilde{p}: \mathrm{Exp}(C|p)\to \mathscr{C}.$ 
\end{definition}

\begin{remark}
    In the literature \cite[§1.1]{Li2001} \cite[§1.4]{MT16}, the family $\mathscr{C}_p\to \mathrm{Exp}(C|p)$ is constructed as follows. {First, performing iterated blow-ups of the constant family $C\times \mathbb{A}^n\to \mathbb{A}^n$ produces a family $\mathbb{C}_n\to \mathbb{A}^n:$ for instance, the family $\mathbb{C}_1\to \mathbb{A}^1$ is given by $\mathrm{Bl}_{p\times 0}(C\times \mathbb{A}^1)\to \mathbb{A}^1,$ and the expansion with a single bubble arises as the special fiber. Observing that $\mathbb{C}_n\to \mathbb{A}^n$ are $\mathbb{G}_m^n$-equivariant, the universal family is then the colimit of $[\mathbb{C}_n/\mathbb{G}_m^n]\to [\mathbb{A}^n/\mathbb{G}_m^n]$ along the morphisms described earlier in this section.}

    From this point of view, the section $\tilde{p}$ is the proper transform of the constant section $p: \mathbb{A}^n\to C\times \mathbb{A}^n.$
\end{remark}

\begin{definition}
  Given $D = \sum_{j=1}^\ell p_j,$ we define $\mathrm{Exp}(C|D) = \prod_{j=1}^{\ell} \mathrm{Exp}(C|p_j),$ so that $\mathrm{Exp}(C|D)$ admits the product cone complex structure.
\end{definition}

 Recall that each universal expanded curve along $p_j$ admits a natural map $\mathscr{C}_j\to C\times \mathrm{Exp}(C,p_j).$ The universal expanded curve along $D$ is their fibre product.

 \begin{definition}
  The \emph{universal expanded curve along} $D$ is the fiber product  over $C$ of each $\mathscr{C}_{p_j}\to C\times \mathrm{Exp}(C|p_j)$: \[\mathscr{C}:= \mathscr{C}_1\times_C \cdots\times_C \mathscr{C}_\ell\to C\times \left(\prod_{j=1}^\ell \mathrm{Exp}(C,p_j)\right) = C\times \mathrm{Exp}(C|D).\]
 \end{definition}

Concretely, an object in $\mathscr{C}_S$ is a family of curves that is constant on the "core curve $C$", or more precisely, an expansion $\tilde C$ over a scheme $S$ along a divisor $D$ over $S$ is a triangle \[\begin{tikzcd}
	{\tilde{C}} && {{C \times S}} \\
	& S
	\arrow[from=1-1, to=1-3]
	\arrow[from=1-1, to=2-2]
	\arrow[from=1-3, to=2-2]
\end{tikzcd}\] such that each geometric fiber, $\tilde C_s \ra C$ is an expansion along $D$ in the sense of Definition \ref{defn:singleexp}. 

\begin{definition}
    Let $\tilde{C}\xrightarrow{\pi} C$ be an expansion over a scheme $S$ of the curve $C$ along $D=p_1+\dots+ p_\ell$.
    Then there are $\ell$ sections $\tilde p_j: S \ra \tilde C$ obtained via the proper transform of the sections $p_j$ as in \ref{defn:scrCtilp}. For each geometric point $s \in S$ we define the \textit{endpoint} of $\tilde C_s$ along $p_j$ as $\tilde p_j(s)$.
\end{definition}


     

\subsection{Logarithmic Hilbert scheme}

\begin{definition}
  Let $\mathrm{Hilb}^n(\mathscr{C})$ be the stack $\mathsf{Sch}_{/\mathrm{Exp}(C|D)}^{\mathrm{op}}\to \mathsf{Grpd}$ defined by $$(\varphi: S\to \mathrm{Exp}(C|D))\mapsto \{\text{length }n \text{ zero-dimensional subschemes on } \mathscr{C}_S\},$$ where isomorphisms on the right hand side are induced by automorphisms of $\mathscr{C}_S$ relative to $S.$
\end{definition}

\begin{remark}
    By construction, the map $\mathrm{Hilb}^n(\mathscr{C})\to \mathrm{Exp}(C|D)$ is representable, so $\mathrm{Hilb}^n(\mathscr{C})$ is algebraic. Therefore, all the stabilizer groups in $\mathrm{Hilb}^n(\mathscr{C})$ are subgroups of the stabilizers of $\mathrm{Exp}(C|D),$ which are algebraic tori. In particular, the finite stabilizer groups in $\mathrm{Hilb}^n(\mathscr{C})$ are products of groups of roots of unity.
\end{remark}

Concretely, a geometric point in $\mathrm{Hilb}^n(\mathscr{C})$ is given by a pair $(\tilde{C},Z),$ where $\pi: \tilde{C}\to C$ is an expanded degeneration along $D,$ and $Z\subset \tilde{C}$ is a length $n$ closed subscheme. The stabilizer group of $(\tilde{C},Z)$ is the subgroup of the rubber torus associated to $\pi: \tilde{C}\to C$ that preserves $Z\subset\tilde{C}.$ 

\begin{notation}
    Given a curve with logarithmic structure, we may use the term \emph{interior} to denote the open subset of points with trivial characteristic monoid relative to the base and the term \emph{special point} to denote a point on the curve with non-trivial logarithmic structure - this can either be the boundary divisor of a smooth curve, a node, or the toric boundary of a bubble.
\end{notation}

\begin{definition}
  A point $(\tilde{C},Z)\in \mathrm{Hilb}^n(\mathscr{C})$ is \textit{logarithmically flat} if $Z$ has no support on any special point of $\tilde{C}$. $(\tilde{C}\xrightarrow{\pi} C,Z)$ is \textit{stable} if each connected component of the interior supports $Z$.

  The \emph{logarithmic Hilbert scheme} of $(C,D),$ denoted as $\Hilb^n(C|D),$ is defined as the open substack of $\mathrm{Hilb}^n(\mathscr{C})$ consisting of stable, logarithmically flat subschemes.
\end{definition}

\section{Construction of $\Hilb^n(C|D)_{\le {\text{\textbf{\textit{i}}}}}$}\label{sec:interm}

\subsection{Intermediate stability conditions} We will construct the intermediate moduli spaces $\Hilb^n(C|D)_{\le \bi}$ by specifying how many points can be supported on the \textit{endpoints} of the bubbles, namely the sections $\tilde{p}_i$ from Definition \ref{defn:scrCtilp}.

\begin{definition}
    Let $\bi=(i_1,\dots, i_\ell) \in \mathbb{Z}_{\geq 0}^\ell.$ We define $\Hilb^n(C|D)_{\le \bi}$ as the subfunctor of the relative Hilbert scheme $\Hilb^n(\mathscr{C}/\mathrm{Exp}(C|D))$ that sends $(f: S\to \mathrm{Exp}(C|D))$ to the groupoid of length $n$ closed subschemes $Z\subset \mathscr{C}_S$ such that: \begin{enumerate}
        \item the subgroup given by $\{\sigma\in \mathrm{Aut}_S(\mathscr{C}_S\to C\times S)\mid \sigma(Z) = Z\}$ is finite,
        \item the open subscheme $Z\smallsetminus \tilde{p}(\mathrm{Exp}(C|D))_S$ is supported away from the nodes\footnote{Equipping $Z$ with the pullback log structure from $\mathscr{C}\to \mathrm{Exp}(C|D),$ this is the same as asking that $Z$ is logarithmically flat over $S.$},
        \item over each geometric point $s\in S,$ $Z_s\subset \mathscr{C}_s$ has multiplicity at most $i_j$ at the endpoint $\tilde{p}_j(s),$
        \item over each geometric point $s\in S,$ let $\mathscr{C}^{(j)}_s\subset \mathscr{C}_s$ be the component of $\mathscr{C}_s$ that contains $\tilde{p}_j(s),$ then $Z_s\cap \mathscr{C}^{(j)}_s\subset \mathscr{C}_s$ has length greater or equal to $i_j+1.$
    \end{enumerate}
The stack $\Hilb^n(C|D)_{\leq \bi}$ admits a universal object as a pair $(\mathscr{C}_{\le \bi}, \mathscr{Z}_{\le \bi})$ with natural maps $\mathscr{Z}_{\le \bi}\hookrightarrow \mathscr{C}_{\le \bi}\to \mathscr{C}$ over $\Exp(C|p).$
\end{definition}

\begin{figure}[htp]
\resizebox{.45\textwidth}{!}{

\tikzset{every picture/.style={line width=0.75pt}} 

\begin{tikzpicture}[x=0.75pt,y=0.75pt,yscale=-1,xscale=1]

\draw [line width=0.75]    (102,50) -- (406,50) ;
\draw [color={rgb, 255:red, 74; green, 144; blue, 226 }  ,draw opacity=1 ]   (151,39) -- (100,121) ;
\draw [color={rgb, 255:red, 74; green, 144; blue, 226 }  ,draw opacity=1 ]   (101,90) -- (150,171) ;
\draw [color={rgb, 255:red, 74; green, 144; blue, 226 }  ,draw opacity=1 ]   (250,40) -- (199,122) ;
\draw [color={rgb, 255:red, 74; green, 144; blue, 226 }  ,draw opacity=1 ]   (380,40) -- (329,122) ;
\draw  [color={rgb, 255:red, 0; green, 0; blue, 0 }  ,draw opacity=1 ][fill={rgb, 255:red, 0; green, 0; blue, 0 }  ,fill opacity=1 ] (145,160.5) .. controls (145,161.33) and (144.33,162) .. (143.5,162) .. controls (142.67,162) and (142,161.33) .. (142,160.5) .. controls (142,159.67) and (142.67,159) .. (143.5,159) .. controls (144.33,159) and (145,159.67) .. (145,160.5) -- cycle ;
\draw  [color={rgb, 255:red, 0; green, 0; blue, 0 }  ,draw opacity=1 ][fill={rgb, 255:red, 0; green, 0; blue, 0 }  ,fill opacity=1 ] (208,110.5) .. controls (208,111.33) and (207.33,112) .. (206.5,112) .. controls (205.67,112) and (205,111.33) .. (205,110.5) .. controls (205,109.67) and (205.67,109) .. (206.5,109) .. controls (207.33,109) and (208,109.67) .. (208,110.5) -- cycle ;
\draw  [color={rgb, 255:red, 208; green, 2; blue, 27 }  ,draw opacity=1 ][fill={rgb, 255:red, 208; green, 2; blue, 27 }  ,fill opacity=1 ] (128.47,126.73) -- (129.27,127.53) -- (126.3,130.5) -- (129.27,133.47) -- (128.47,134.27) -- (125.5,131.3) -- (122.53,134.27) -- (121.73,133.47) -- (124.7,130.5) -- (121.73,127.53) -- (122.53,126.73) -- (125.5,129.7) -- cycle ;
\draw  [color={rgb, 255:red, 208; green, 2; blue, 27 }  ,draw opacity=1 ][fill={rgb, 255:red, 208; green, 2; blue, 27 }  ,fill opacity=1 ] (341.27,109.91) -- (341.28,111.04) -- (337.07,111.06) -- (337.09,115.27) -- (335.96,115.28) -- (335.94,111.07) -- (331.73,111.09) -- (331.72,109.96) -- (335.93,109.94) -- (335.91,105.73) -- (337.04,105.72) -- (337.06,109.93) -- cycle ;
\draw  [color={rgb, 255:red, 208; green, 2; blue, 27 }  ,draw opacity=1 ][fill={rgb, 255:red, 208; green, 2; blue, 27 }  ,fill opacity=1 ] (339.42,106.67) -- (340.22,107.47) -- (337.29,110.49) -- (340.27,113.47) -- (339.48,114.28) -- (336.51,111.31) -- (333.58,114.33) -- (332.78,113.53) -- (335.71,110.51) -- (332.73,107.53) -- (333.52,106.72) -- (336.49,109.69) -- cycle ;
\draw  [fill={rgb, 255:red, 0; green, 0; blue, 0 }  ,fill opacity=1 ] (338,110.5) .. controls (338,109.67) and (337.33,109) .. (336.5,109) .. controls (335.67,109) and (335,109.67) .. (335,110.5) .. controls (335,111.33) and (335.67,112) .. (336.5,112) .. controls (337.33,112) and (338,111.33) .. (338,110.5) -- cycle ;

\draw  [color={rgb, 255:red, 208; green, 2; blue, 27 }  ,draw opacity=1 ][fill={rgb, 255:red, 208; green, 2; blue, 27 }  ,fill opacity=1 ] (209.47,106.73) -- (210.27,107.53) -- (207.3,110.5) -- (210.27,113.47) -- (209.47,114.27) -- (206.5,111.3) -- (203.53,114.27) -- (202.73,113.47) -- (205.7,110.5) -- (202.73,107.53) -- (203.53,106.73) -- (206.5,109.7) -- cycle ;
\draw  [fill={rgb, 255:red, 0; green, 0; blue, 0 }  ,fill opacity=1 ] (208,110.5) .. controls (208,109.67) and (207.33,109) .. (206.5,109) .. controls (205.67,109) and (205,109.67) .. (205,110.5) .. controls (205,111.33) and (205.67,112) .. (206.5,112) .. controls (207.33,112) and (208,111.33) .. (208,110.5) -- cycle ;

\draw  [color={rgb, 255:red, 208; green, 2; blue, 27 }  ,draw opacity=1 ][fill={rgb, 255:red, 208; green, 2; blue, 27 }  ,fill opacity=1 ] (360.45,73.45) -- (361.25,74.25) -- (358.27,77.23) -- (361.25,80.2) -- (360.45,81) -- (357.47,78.03) -- (354.5,81) -- (353.7,80.2) -- (356.67,77.23) -- (353.7,74.25) -- (354.5,73.45) -- (357.47,76.42) -- cycle ;
\draw  [color={rgb, 255:red, 208; green, 2; blue, 27 }  ,draw opacity=1 ][fill={rgb, 255:red, 208; green, 2; blue, 27 }  ,fill opacity=1 ] (228.27,77.23) -- (229.08,78.03) -- (226.1,81) -- (229.08,83.97) -- (228.27,84.77) -- (225.3,81.8) -- (222.33,84.77) -- (221.53,83.97) -- (224.5,81) -- (221.53,78.03) -- (222.33,77.23) -- (225.3,80.2) -- cycle ;
\draw  [color={rgb, 255:red, 208; green, 2; blue, 27 }  ,draw opacity=1 ][fill={rgb, 255:red, 208; green, 2; blue, 27 }  ,fill opacity=1 ] (139.5,57.95) -- (140.3,58.75) -- (137.33,61.73) -- (140.3,64.7) -- (139.5,65.5) -- (136.53,62.53) -- (133.55,65.5) -- (132.75,64.7) -- (135.73,61.73) -- (132.75,58.75) -- (133.55,57.95) -- (136.53,60.92) -- cycle ;
\draw  [color={rgb, 255:red, 208; green, 2; blue, 27 }  ,draw opacity=1 ][fill={rgb, 255:red, 208; green, 2; blue, 27 }  ,fill opacity=1 ] (125.5,80) -- (126.3,80.8) -- (123.33,83.77) -- (126.3,86.75) -- (125.5,87.55) -- (122.53,84.58) -- (119.55,87.55) -- (118.75,86.75) -- (121.73,83.77) -- (118.75,80.8) -- (119.55,80) -- (122.53,82.97) -- cycle ;
\draw  [color={rgb, 255:red, 208; green, 2; blue, 27 }  ,draw opacity=1 ][fill={rgb, 255:red, 208; green, 2; blue, 27 }  ,fill opacity=1 ] (188.5,45.95) -- (189.3,46.75) -- (186.33,49.73) -- (189.3,52.7) -- (188.5,53.5) -- (185.53,50.53) -- (182.55,53.5) -- (181.75,52.7) -- (184.73,49.73) -- (181.75,46.75) -- (182.55,45.95) -- (185.53,48.92) -- cycle ;
\draw  [color={rgb, 255:red, 0; green, 0; blue, 0 }  ,draw opacity=1 ][fill={rgb, 255:red, 0; green, 0; blue, 0 }  ,fill opacity=1 ] (146,49.5) .. controls (146,50.33) and (145.33,51) .. (144.5,51) .. controls (143.67,51) and (143,50.33) .. (143,49.5) .. controls (143,48.67) and (143.67,48) .. (144.5,48) .. controls (145.33,48) and (146,48.67) .. (146,49.5) -- cycle ;
\draw  [color={rgb, 255:red, 0; green, 0; blue, 0 }  ,draw opacity=1 ][fill={rgb, 255:red, 0; green, 0; blue, 0 }  ,fill opacity=1 ] (246,49.5) .. controls (246,50.33) and (245.33,51) .. (244.5,51) .. controls (243.67,51) and (243,50.33) .. (243,49.5) .. controls (243,48.67) and (243.67,48) .. (244.5,48) .. controls (245.33,48) and (246,48.67) .. (246,49.5) -- cycle ;
\draw  [color={rgb, 255:red, 0; green, 0; blue, 0 }  ,draw opacity=1 ][fill={rgb, 255:red, 0; green, 0; blue, 0 }  ,fill opacity=1 ] (375,50.5) .. controls (375,51.33) and (374.33,52) .. (373.5,52) .. controls (372.67,52) and (372,51.33) .. (372,50.5) .. controls (372,49.67) and (372.67,49) .. (373.5,49) .. controls (374.33,49) and (375,49.67) .. (375,50.5) -- cycle ;
\draw  [color={rgb, 255:red, 208; green, 2; blue, 27 }  ,draw opacity=1 ][fill={rgb, 255:red, 208; green, 2; blue, 27 }  ,fill opacity=1 ] (306.47,46.73) -- (307.27,47.53) -- (304.3,50.5) -- (307.27,53.47) -- (306.47,54.27) -- (303.5,51.3) -- (300.53,54.27) -- (299.73,53.47) -- (302.7,50.5) -- (299.73,47.53) -- (300.53,46.73) -- (303.5,49.7) -- cycle ;
\draw  [fill={rgb, 255:red, 0; green, 0; blue, 0 }  ,fill opacity=1 ] (305,50.5) .. controls (305,49.67) and (304.33,49) .. (303.5,49) .. controls (302.67,49) and (302,49.67) .. (302,50.5) .. controls (302,51.33) and (302.67,52) .. (303.5,52) .. controls (304.33,52) and (305,51.33) .. (305,50.5) -- cycle ;

\draw (135,24) node [anchor=north west][inner sep=0.75pt]  [font=\normalsize] [align=left] {$\displaystyle p_{1}$};
\draw (235,24) node [anchor=north west][inner sep=0.75pt]   [align=left] {$\displaystyle p_{2}$};
\draw (295,24) node [anchor=north west][inner sep=0.75pt]   [align=left] {$\displaystyle p_{3}$};
\draw (363,24) node [anchor=north west][inner sep=0.75pt]   [align=left] {$\displaystyle p_{4}$};

\end{tikzpicture}
}
\caption{An example of a point in $\Hilb^{10}(C|D)_{\le \boldsymbol{i}}$} with $\boldsymbol{i}=(0,1,1,2)$. Note that it would not be an example for $\boldsymbol{i}=(1,1,1,2)$ or $\boldsymbol{i}=(0,2,1,2)$.
\label{fig:intHilb}
\end{figure} 

\begin{remark}
When $\bi=(0,\dots, 0)$ we recover the logarithmic Hilbert scheme. For any $\bi\ge (n,\dots ,n)$, we recover the classical Hilbert scheme $\Hilb^n(C)$.
\end{remark}

\begin{remark}
    $\Hilb^n(C|D)_{\leq \bi}$ is an open substack of $\Hilb^n(\mathscr{C}/\Exp(C|D))$ defined by the above conditions. The condition $(4)$ ensures separatedness of the moduli space. 
\end{remark}

\subsection{$\textbf{\text{Hilb}}^{\boldsymbol{n}}{(\boldsymbol{C}|\boldsymbol{D})}_{\leq \bi}$ and Hassett spaces}\label{subsec:Hilbquot}

In this section, we relate $\Hilb^n(C|D)_{\le \bi}$ to quotient stacks of Hassett spaces with heavy-light weights by the action of the symmetric group.

We will consider the Hassett space $\Mbar_{g,\mathcal{A}}$ with $\mathcal{A} \in \mathbb{Q}_{\ge 0}^{\ell + n},$ where the endpoints will be the first $\ell$ marked points ("heavy points"), the other $n$ marked points ("light points") are permuted by $S_n.$ An $S_n$-orbit in $\Mbar_{g, \cA}$ corresponds to a geometric point in $\Hilb^n(C|D)_{\leq \bi}.$

\begin{definition}\label{defn:wtAbi}
    Let $\bi=(i_1,\dots,i_\ell) \in (\mathbb{Z}\cap [0,n])^\ell$ and define the set of weights \begin{equation} \label{hassettweights}\mathcal{A}:=\left(w_1, \dots ,w_\ell, \delta_n, \dots, \delta_n\right) \in (\mathbb{Q} \cap (0,1])^{\ell +n}\end{equation} with $w_j=1-i_j\delta_n$ and $0<\delta_n<\frac{1}{2n}.$ 
\end{definition}

Note that $\delta_n$ is small enough that: \begin{enumerate}
    \item $n\delta_n < 1$, which implies all the light points can collide away from the heavy points;
    \item $w_j + w_k>1 \ $ for each $  1\le j<k\le \ell$, which implies no two heavy points are allowed to collide;
    \item $i_j\delta_n+w_j=1$ and $(i_j+1)\delta_n + w_j>1$, which implies at most $i_j$ many light points will be allowed to lie on the $j$-th heavy point.
\end{enumerate}

 Let $\mathcal{A'}:=\left(w_1, \dots , w_\ell \right) \in (\mathbb{Q} \cap (0,1])^{\ell}$ be the vector consisting of only heavy weights. 

\begin{remark}
    The weights $\cA$ and $\cA'$ depend on the vector $\bi.$ We omit this from the notation when there is no risk of confusion.
\end{remark}

By \cite[Theorem 4.3]{hassett} there exists a forgetful morphism of stacks $$\phi_{\mathcal{A}, \mathcal{A}'}: \Mbar_{g,\mathcal{A}} \ra \Mbar_{g,\mathcal{A}'}$$ where given a point $(\tilde C,\tilde p_1 \dots \tilde p_\ell, q_1,\dots,q_n) \in \Mbar_{g, \cA}$, the curve $\phi_{\mathcal{A}, \mathcal{A}'}(\tilde C,\tilde p_1 \dots \tilde p_\ell, q_1,\dots,q_n)=(C, p_1,\dots,p_l)$ is obtained by successively collapsing components of $\tilde{C}$ along which $K_{\tilde C}+ w_1\tilde p_1+\dots w_\ell \tilde p_\ell$ fails to be ample.

Moreover, by \cite[Proposition 5.1]{hassett}, the weights $\mathcal{A}'$ and $\mathbf{1} = (1,\dots,1)\in \mathbb{Q}^{\ell}$ are in the same chamber, so there is an isomorphism $\Mbar_{g,\mathcal{A'}} \cong \Mbar_{g,\ell}$. We will denote with $\phi_{\mathcal{A}}: \Mbar_{g,\mathcal{A}} \ra \Mbar_{g,\ell}$ the induced morphism.

Let us focus our attention on the preimage of $(C,p_1,\dots,p_\ell)$. 

\begin{lemma}\label{lem:Aexp}
    Let $(\tilde C,\tilde p_1, \dots, \tilde p_\ell, q_1,\dots,q_n) \in \phi_{\mathcal{A}}^{-1}(C,p_1,\dots, p_\ell)$ for $(C,p_1\dots,p_\ell) \in \M_{g,\ell}$. Then the underlying curve $(\tilde C, \tilde p_1, \dots, \tilde p_\ell)$ is an expansion in $\Exp(C|D)$ with $D=p_1+\dots p_\ell$.
\end{lemma}

\begin{proof} 
Consider $R^{(i)}:=\phi_{\mathcal{A}}^{-1}(p_i)\subset \tilde{C}.$ This is a connected subcurve of $\tilde{C}$ that gets contracted to $p_i$ under the map $\phi_{\mathcal{A}}.$ Because both $\tilde{C}$ and $C$ have the same arithmetic genus, $R^{(i)}$ must be a tree of rational curves. To show that $\tilde{C}\to C$ is an expansion along $D,$ it suffices to show that each $R^{(i)}$ is a chain of rational curves. If this is true, then the curve $(\tilde{C},\tilde{p}_1,\dots,\tilde{p}_\ell)$ defines a point in $\mathrm{Exp}(C|D).$

If the subcurve $R^{(i)}$ is connected to the rest of $\tilde{C}$ via more than one point, then the dual graph of $\tilde{C}$ contains a loop, and so the arithmetic genus of $\tilde{C}$ is strictly greater than $g,$ contradiction. Thus, let $\overline{p}_i\in R^{(i)}$ be the unique node that connects $R^{(i)}$ with the rest of $\tilde{C}.$ The curve $R^{(i)}$ with the marked points of $\tilde{C}$ that are on $R^{(i)}$ and $\overline{p}_i$ is a weighted stable curve, where $\overline{p}_i$ is assigned weight 1. By construction, $\phi_{\mathcal{A}}^{-1}(\tilde{p}_i)$ are disjoint for distinct $i,$ so among the marked points of $\tilde{C}$ that are on $R^{(i)},$ the only heavy point is $\tilde{p}_i.$ Therefore, the only two heavy points on the weighted stable curve $R^{(i)}$ are $\tilde{p}_i$ and $\overline{p}_i,$ and the rest are light points. Now the combinatorics of the Losev--Manin spaces implies that $R^{(i)}$ must be a chain of rational curves.
\end{proof}

\begin{remark}\label{rem:Aexp}
    We explain how the point-wise statement in the lemma lifts to the level of moduli stacks, namely that there is a well-defined map $\phi^{-1}_{\mathcal{A}}(C, p_1,\dots,p_\ell)\to \mathrm{Exp}(C|D).$ Indeed, let $\mathfrak{M}_{g, \ell}$ be the stack of nodal curves of arithmetic genus $g$ and $\ell$ marked points and let $\phi:\mathfrak{M}_{g,\ell}\to \Mbar_{g, \ell}$ be the stabilization map. Because $(C,p_1,\dots, p_\ell) \in \mathcal{M}_{g,\ell}$,  the fiber $\Mbar_C:=\phi^{-1}(C,p_1,\dots, p_\ell)$ is the stack of all nodal genus $g$ curves that are given by attaching trees of rational components (possibly unstable) to $C$ along the points $p_i$. Let $\Mbar_C^{\mathrm{path}}\subset \Mbar_{C}$ be the locally closed substack where each union of rational components is a path. By construction, $\Mbar_C^{\mathrm{path}}$ parametrizes expansions of $C$ along $D,$ and their automorphism groups match by Remark \ref{rem:autgroups}, so $\Mbar_C^{\mathrm{path}}\cong \mathrm{Exp}(C|D).$ On the other hand, by definition the preimage $\phi_{\mathcal{A}}^{-1}(C, p_1,\dots,p_\ell)$ maps to $\Mbar_{C},$ and the above lemma proves that the map to $\Mbar_{C}$ lands in $\Mbar_{C}^{\mathrm{path}},$ hence there is a map of stacks $\phi_{\mathcal{A}}^{-1}(C, p_1,\dots,p_\ell)\to \mathrm{Exp}(C|D).$
    
    We note that a more straightforward way of defining the morphism is to use the characterization of $\mathrm{Exp}(\mathbb{A}^1|0)$ as an open substack in $\mathfrak{M}_{0,3}^{ss}$ and natural forgetful morphisms between moduli stacks of curves \cite[§2.5]{GraberVakil} \cite[§3.1]{ACFW}.
\end{remark}

\begin{remark}
    Note that in \cite[Therem 4.3]{hassett}, Hassett defines the forgetful maps $\phi_{\cA, \cA'}$ only for $2g-2+\ell>0$. Nonetheless, such maps can be defined even in the case of $C = \mathbb{P}^1$ and $|D| = 1,2$. In this case $\Mbar_{g,\mathcal{A}}$ and $\Mbar_{g,\ell}$ should be replaced by appropriate Artin stacks, and the forgetful map is still well-defined and representable, with fibers being schemes.
\end{remark}

\begin{proposition} \label{prop:quotient} Let $\bi=(i_1,\dots,i_\ell) \in (\mathbb{Z}\cap [0,n])^\ell$ and $\mathcal{A}$ defined as in (\ref{hassettweights}). Then there is a natural map \[[\phi_{\mathcal{A}}^{-1}(C,p_1,\dots,p_\ell)/S_n]\to \Hilb^n(C|D)_{\leq \bi}\]
over $\mathrm{Exp}(C|D)$ that is a relative coarse moduli space map\footnote{This is also termed `coarsening morphism' by David Rydh \cite[Appendix A]{ATWR}.}.
\end{proposition}

\begin{proof}
    The target $\Hilb^n(C|D)_{\leq \bi}$ represents the moduli problem of pairs $(\tilde{C},p_1,\dots,p_\ell, Z),$ where: \begin{enumerate}
        \item  $\tilde{C}$ is a nodal curve of arithmetic genus $g,$
        \item $p_1,\dots,p_\ell\in \tilde{C}$ are distinct smooth points,
        \item $Z$ is a length-$n$ divisor supported on the smooth locus of $\tilde{C}$ \end{enumerate} such that $(\tilde{C}, p_1,\dots,p_\ell, Z)$ satisfies the symmetrized $\mathcal{A}$-stability condition--the weights $\mathcal{A}$ set up in Section \ref{subsec:Hilbquot} is $S_n$-invariant. Therefore, the universal family over $\phi_{\mathcal{A}}^{-1}(C,p_1,\dots,p_\ell)$ defines a map to $\Hilb^n(C|D)_{\leq \bi}.$ Since the map is $S_n$-invariant, it descends to a map from the $S_n$-quotient stack $[\phi_{\mathcal{A}}^{-1}(C,p_1,\dots,p_\ell)/S_n]\to \Hilb^n(C|D)_{\leq \bi}$ that induces a bijection on geometric points.

        To show that this map of stacks is a relative coarse moduli space, for any morphism $S\to \Exp(C|D),$ we pull back the universal family $\mathscr{C}\to \mathrm{Exp}(C|D)$ to get a family of curves $\mathscr{C}_S\to S.$ Let $\mathscr{C}_S^{\mathrm{sm}}\subset \mathscr{C}_S$ be the smooth locus of the family. Using $(\mathscr{C}^{\mathrm{sm}})^n_S$ resp. $\mathrm{Sym}^{n}\mathscr{C}^{\mathrm{sm}}_S$ to denote the $n$-th fiber power resp. symmetric power relative to $S,$ the base change of the morphism $[\phi_{\mathcal{A}}^{-1}(C,p_1,\dots,p_\ell)/S_n]\to \Hilb^n(C|D)_{\leq \bi}$ is the relative coarse moduli space morphism $[(\mathscr{C}^{\mathrm{sm}})^n_S/S_n]\to \mathrm{Sym}^{n}\mathscr{C}^{\mathrm{sm}}_S.$ Applying the argument to the smooth chart $\prod_{i=1}^{\ell}(\bigsqcup_{n}\mathbb{A}^n)\to \Exp(C|D)$ (see for instance \cite[§1]{MT16}), we see that the morphism satisfies the definition of a coarse moduli space map given in \cite[Theorem 3.1]{AOV}.

\end{proof}

    

We recall the following notation for the weight reduction maps that will be revisited in Section \ref{sec:inducstep}.

\begin{definition}[Reduction maps]
    Let $\bi_1, \bi_2\in (\mathbb{Z}\cap [0,n])^\ell$ such that $\bi_1\leq \bi_2,$ and let $\mathcal{A}_{\bi_1}, \mathcal{A}_{\bi_2}$ be the associated weights from Definition \ref{defn:wtAbi}. Denote $$\rho_{\mathcal{A}_{\bi_1},\mathcal{A}_{\bi_2}}: \Mbar_{g, \mathcal{A}_{\bi_1}}\to \Mbar_{g, \mathcal{A}_{\bi_2}}$$ as the weight reduction map of Hassett spaces \cite[Theorem 4.1]{hassett}. Abusing notation, we also denote its restriction to $\phi_{\mathcal{A}_{\bi_1}}^{-1}(C, p_1,\dots,p_\ell)\to \phi_{\mathcal{A}_{\bi_2}}^{-1}(C, p_1,\dots,p_\ell)$ as $\rho_{\mathcal{A}_{\bi_1},\mathcal{A}_{\bi_2}}$ or its shorthand $\rho.$

    The symmetric group $S_n$ acts on both $\phi_{\mathcal{A}_{\bi_i}}^{-1}(C, p_1,\dots,p_\ell)$ ($i = 1,2$) by permuting the light points, and $\rho$ is $S_n$-equivariant, so the map descends to a map of the $S_n$-quotient stacks and then to their relative coarse moduli spaces $\Hilb^n(C|D)_{\leq \bi_1}\to \Hilb^n(C|D)_{\leq \bi_2}.$
\end{definition}

\subsection{$\textbf{\text{Hilb}}^{\boldsymbol{n}}{(\boldsymbol{C}|\boldsymbol{D})}_{\leq \bi}$ is a smooth Deligne-Mumford stack}

\begin{proposition}\label{cor:Hilbsm}
    The intermediate moduli spaces $\Hilb^n(C| D)_{\leq \bi}$ are smooth DM stacks of dimension $n.$
\end{proposition}

\begin{proof}

    Since $\prod_{i=1}^{\ell}(\bigsqcup_{n}\mathbb{A}^n)\to \Exp(C|D)$ form a smooth chart of $\mathrm{Exp}(C|D),$ the stack $\mathrm{Exp}(C|D)$ is smooth. It suffices to show that the map $\Hilb^n(C| D)_{\leq \bi}\to \mathrm{Exp}(C|D)$ is smooth. The smoothness of the morphism follows from the description in the proof of Proposition \ref{prop:quotient} that given any morphism $S\to \mathrm{Exp}(C|D),$ the pullback is an open immersion of $\mathrm{Sym}^{n}\mathscr{C}^{\mathrm{sm}}_S\to S,$ which is smooth.
\end{proof}

\section{The example case of $\Hilb^2(\PP^1|0)$}\label{sec:2P10}

This section introduces the concrete example of two points on the pair $(\PP^1|0).$ Complementary to our strategy of proving the main results, we have chosen to perform weighted blow-up on $\Hilb^2(\PP^1)$ in this section, as it motivates the presence of the weights and illustrates the toric structure of the moduli space more effectively.

\subsection{$\textbf{\text{Hilb}}^{\boldsymbol{2}}{\boldsymbol{(\PP^1)}}$ as a toric variety}
There is an isomorphism $\Hilb^2(\PP^1)\cong\PP^2$ given by degree two homogeneous polynomials on $\PP^1$ with a basis of symmetric polynomials:
$$\{[a_1: b_1], [a_2: b_2]\}\mapsto [a_1a_2:-(a_1b_2+a_2b_1): b_1b_2].$$

The toric boundary of $\PP^2$ corresponds to the following loci in $\Hilb^2(\PP^1)$: 
{The toric boundary of $\PP^2$ corresponds to the loci where one point lies on $0,$ one point lies on $\infty,$ and the anti-diagonal $a_1b_2 + a_2b_1 = 0.$}

The toric strata and their specializations are illustrated in Figure \ref{fig:P2}. 
\input{P2}

\subsection{First weighted blow-up}

{By definition, the moduli space $\Hilb(\PP^1|0)_{\le 1}$ parametrizes points on expansions of $\PP^1$, where we allow at most one point to be supported on the endpoint. Examples are pictured in Figure \ref{fig:pointsmoduli}. We recall that the points supported on the bubble are considered up to the action of the rubber torus action, which corresponds to the dense torus of the bubble $\PP^1$.}
\begin{figure}[htp]
\resizebox{.7\textwidth}{!}{

\tikzset{every picture/.style={line width=0.75pt}} 

\begin{tikzpicture}[x=0.75pt,y=0.75pt,yscale=-1,xscale=1]

\draw [color={rgb, 255:red, 74; green, 144; blue, 226 }  ,draw opacity=1 ]   (106,26) -- (84,85) ;
\draw [color={rgb, 255:red, 0; green, 0; blue, 0 }  ,draw opacity=1 ]   (88,40) -- (171,40) ;
\draw  [fill={rgb, 255:red, 0; green, 0; blue, 0 }  ,fill opacity=1 ] (163,39.5) .. controls (163,38.67) and (162.33,38) .. (161.5,38) .. controls (160.67,38) and (160,38.67) .. (160,39.5) .. controls (160,40.33) and (160.67,41) .. (161.5,41) .. controls (162.33,41) and (163,40.33) .. (163,39.5) -- cycle ;
\draw  [fill={rgb, 255:red, 0; green, 0; blue, 0 }  ,fill opacity=1 ] (102,40.5) .. controls (102,39.67) and (101.33,39) .. (100.5,39) .. controls (99.67,39) and (99,39.67) .. (99,40.5) .. controls (99,41.33) and (99.67,42) .. (100.5,42) .. controls (101.33,42) and (102,41.33) .. (102,40.5) -- cycle ;
\draw  [color={rgb, 255:red, 208; green, 2; blue, 27 }  ,draw opacity=1 ][fill={rgb, 255:red, 208; green, 2; blue, 27 }  ,fill opacity=1 ] (97.97,50.92) -- (98.77,51.73) -- (95.8,54.7) -- (98.77,57.67) -- (97.97,58.47) -- (95,55.5) -- (92.03,58.47) -- (91.23,57.67) -- (94.2,54.7) -- (91.23,51.73) -- (92.03,50.92) -- (95,53.9) -- cycle ;
\draw [color={rgb, 255:red, 74; green, 144; blue, 226 }  ,draw opacity=1 ]   (208,26) -- (186,85) ;
\draw  [color={rgb, 255:red, 208; green, 2; blue, 27 }  ,draw opacity=1 ][fill={rgb, 255:red, 208; green, 2; blue, 27 }  ,fill opacity=1 ] (199.97,50.92) -- (200.77,51.73) -- (197.8,54.7) -- (200.77,57.67) -- (199.97,58.47) -- (197,55.5) -- (194.03,58.47) -- (193.23,57.67) -- (196.2,54.7) -- (193.23,51.73) -- (194.03,50.92) -- (197,53.9) -- cycle ;
\draw [color={rgb, 255:red, 0; green, 0; blue, 0 }  ,draw opacity=1 ]   (190,40) -- (273,40) ;
\draw  [fill={rgb, 255:red, 0; green, 0; blue, 0 }  ,fill opacity=1 ] (265,39.5) .. controls (265,38.67) and (264.33,38) .. (263.5,38) .. controls (262.67,38) and (262,38.67) .. (262,39.5) .. controls (262,40.33) and (262.67,41) .. (263.5,41) .. controls (264.33,41) and (265,40.33) .. (265,39.5) -- cycle ;
\draw  [fill={rgb, 255:red, 0; green, 0; blue, 0 }  ,fill opacity=1 ] (204,40.5) .. controls (204,39.67) and (203.33,39) .. (202.5,39) .. controls (201.67,39) and (201,39.67) .. (201,40.5) .. controls (201,41.33) and (201.67,42) .. (202.5,42) .. controls (203.33,42) and (204,41.33) .. (204,40.5) -- cycle ;
\draw  [color={rgb, 255:red, 208; green, 2; blue, 27 }  ,draw opacity=1 ][fill={rgb, 255:red, 208; green, 2; blue, 27 }  ,fill opacity=1 ] (194.47,66.42) -- (195.27,67.23) -- (192.3,70.2) -- (195.27,73.17) -- (194.47,73.97) -- (191.5,71) -- (188.53,73.97) -- (187.73,73.17) -- (190.7,70.2) -- (187.73,67.23) -- (188.53,66.42) -- (191.5,69.4) -- cycle ;
\draw [color={rgb, 255:red, 0; green, 0; blue, 0 }  ,draw opacity=1 ]   (290,40) -- (373,40) ;
\draw  [fill={rgb, 255:red, 0; green, 0; blue, 0 }  ,fill opacity=1 ] (365,39.5) .. controls (365,38.67) and (364.33,38) .. (363.5,38) .. controls (362.67,38) and (362,38.67) .. (362,39.5) .. controls (362,40.33) and (362.67,41) .. (363.5,41) .. controls (364.33,41) and (365,40.33) .. (365,39.5) -- cycle ;
\draw  [color={rgb, 255:red, 208; green, 2; blue, 27 }  ,draw opacity=1 ][fill={rgb, 255:red, 208; green, 2; blue, 27 }  ,fill opacity=1 ] (305.27,39.91) -- (305.28,41.04) -- (301.07,41.06) -- (301.09,45.27) -- (299.96,45.28) -- (299.94,41.07) -- (295.73,41.09) -- (295.72,39.96) -- (299.93,39.94) -- (299.91,35.73) -- (301.04,35.72) -- (301.06,39.93) -- cycle ;
\draw  [color={rgb, 255:red, 208; green, 2; blue, 27 }  ,draw opacity=1 ][fill={rgb, 255:red, 208; green, 2; blue, 27 }  ,fill opacity=1 ] (303.42,36.67) -- (304.22,37.47) -- (301.29,40.49) -- (304.27,43.47) -- (303.48,44.28) -- (300.51,41.31) -- (297.58,44.33) -- (296.78,43.53) -- (299.71,40.51) -- (296.73,37.53) -- (297.52,36.72) -- (300.49,39.69) -- cycle ;
\draw  [fill={rgb, 255:red, 0; green, 0; blue, 0 }  ,fill opacity=1 ] (302,40.5) .. controls (302,39.67) and (301.33,39) .. (300.5,39) .. controls (299.67,39) and (299,39.67) .. (299,40.5) .. controls (299,41.33) and (299.67,42) .. (300.5,42) .. controls (301.33,42) and (302,41.33) .. (302,40.5) -- cycle ;

\draw  [color={rgb, 255:red, 126; green, 211; blue, 33 }  ,draw opacity=1 ][fill={rgb, 255:red, 126; green, 211; blue, 33 }  ,fill opacity=1 ] (259.43,68.57) -- (266,58) -- (260,70) -- (256,66.86) -- (259.43,68.57) -- cycle ;
\draw  [color={rgb, 255:red, 208; green, 2; blue, 27 }  ,draw opacity=1 ][fill={rgb, 255:red, 208; green, 2; blue, 27 }  ,fill opacity=1 ] (352.16,62.1) -- (356.17,64.59) -- (363.6,60.8) -- (357.32,65.38) -- (360.21,68.33) -- (356.62,65.85) -- (352.02,67.86) -- (355.63,65.19) -- cycle ;
\draw  [color={rgb, 255:red, 126; green, 211; blue, 33 }  ,draw opacity=1 ][fill={rgb, 255:red, 126; green, 211; blue, 33 }  ,fill opacity=1 ] (147.43,68.57) -- (154,58) -- (148,70) -- (144,66.86) -- (147.43,68.57) -- cycle ;
\draw  [color={rgb, 255:red, 208; green, 2; blue, 27 }  ,draw opacity=1 ][fill={rgb, 255:red, 208; green, 2; blue, 27 }  ,fill opacity=1 ] (89.47,73.73) -- (90.27,74.53) -- (87.3,77.5) -- (90.27,80.47) -- (89.47,81.27) -- (86.5,78.3) -- (83.53,81.27) -- (82.73,80.47) -- (85.7,77.5) -- (82.73,74.53) -- (83.53,73.73) -- (86.5,76.7) -- cycle ;
\draw  [fill={rgb, 255:red, 0; green, 0; blue, 0 }  ,fill opacity=1 ] (88,77.5) .. controls (88,76.67) and (87.33,76) .. (86.5,76) .. controls (85.67,76) and (85,76.67) .. (85,77.5) .. controls (85,78.33) and (85.67,79) .. (86.5,79) .. controls (87.33,79) and (88,78.33) .. (88,77.5) -- cycle ;

\end{tikzpicture}
}
\caption{Examples of allowed and disallowed positions for the pair of points in $\Hilb^n(\mathbb{P}^1|0)_{\le 1}$.}
\label{fig:pointsmoduli}
\end{figure}
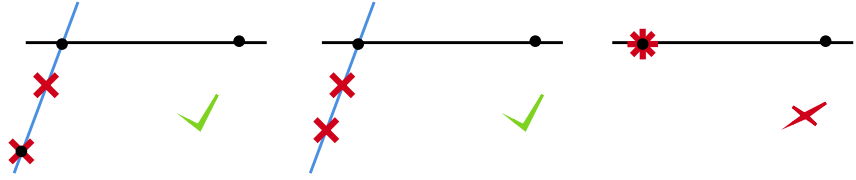 

Consider the locus in $\Hilb^2(\PP^1|0)_{\leq 1}$ where both points are supported on the bubble $\PP^1$. We claim that this locus is a divisor in $\Hilb(\PP^1|0)_{\le 1}$ isomorphic to the weighted projective stack $\Pcal(1,2)$.
Indeed, the new locus parametrizes two points supported the smooth locus of the bubble and not both supported on the endpoint $\tilde 0$. This is the same as a point in $\Sym^2(\PP^1\setminus \infty)\setminus \{2\cdot 0\} = \A^2_{s_1+s_2, s_1s_2} \setminus 0.$ The action of the rubber torus on the bubble induces an action of $\Gm$ on $\A^2_{s_1+s_2, s_1 s_2} \setminus 0$ with weights $1$ and $2$ on the two coordinates. Our claim follows from taking the $\Gm$-quotient $[(\A^2_{s_1+s_2, s_1 s_2} \setminus 0)/\Gm]\cong \cP(1,2).$ We have effectively described performing a weighted blow-up centered at the intersection of the toric boundary divisors $V(a_1b_2+a_2b_1)$ and $V(b_1b_2)$ with weights $1$ and $2$ respectively. 

By blowing up a toric variety at a torus invariant locus, we obtain a toric stack, its cones and specializations are illustrated in Figure \ref{fig:P2B1}.

\input{P2B1}


\subsection{The isomorphism $\textbf{\text{Hilb}}^{\boldsymbol{2}}{(\boldsymbol{\PP^1}|\boldsymbol{0})}_{\leq \mathbf{1}}\cong\textbf{\text{Hilb}}^{\boldsymbol{2}}{(\boldsymbol{\PP^1}|\boldsymbol{0})}$}

Going from $\Hilb^2(\PP^1|0)_{\le 1}$ to $\Hilb^2(\PP^1|0)$ consists of replacing the locus where one point lies on the endpoint by placing the point on a further bubble (cf. Figure \ref{fig:P2B2}). Even though we are changing the moduli problem (i.e. the universal families), we have an induced isomorphism of moduli spaces since the newly introduced bubble with one point is stabilized by $\Gm$. 

\input{P2B2}

\subsection{Stanley--Reisner ring calculation}\label{subsec:SR}
Extending the above discussion, we describe the Chow ring of $\Hilb^n(\mathbb{P}^1|0)_{\leq i}$ using the following corollary of Theorem \ref{thm:main} and Corollary \ref{cor:P1}.  Let us consider the standard lattice and fan for $\PP^n \cong \Sym^n(\PP_{[a:b]}^1)$. Then the ray generated by $e_k$ corresponds to the zero locus of the the symmetric polynomial of degree $k$ in $t=a/b$.

\begin{corollary}[Toric stack description]
    The iterated star subdivision of the fan of $\PP^n$ along the rays $(1,2,\dots, n),$ $(1,2,\dots, n-1,0),\dots,(1,2,\dots,i,\dots, 0)$ gives the fan of $\Hilb^n(\mathbb{P}^1|0)_{\leq i}$ as a toric stack. Following \cite[Example 5.4.2]{quek-rydh-weighted-blow-up}, the stacky structure is therefore given by the map $\beta: \bZ^{\Sigma(1)} \to \bZ^n$ mapping each ray to its primitive ray generator.
\end{corollary}

Therefore, the Chow ring of $\Hilb^n(\mathbb{P}^1|0)_{\leq i}$ can be calculated by the Stanley-Reisner ring of the fan.

\begin{definition}
   We fix coordinates $[a:b]\in \mathbb{P}^1$ as before and set $t=b/a.$ 
   
   For $j=1,\dots,n,$ let $\sigma_{j}$ be the ray corresponding to the coordinate hyperplanes on $\Sym^n(\mathbb{P}^1)\cong \mathbb{P}^n$ associated to the elementary symmetric polynomials $\sum_{\substack{J\subset \{1,\dots,n\}\\ |J|=j}}\prod_{k\in J}t_{k},$ and let $\tau$ be the remaining coordinate hyperplane (associated to $1/\prod_{i=1}^n t_i).$ We assign primitive vectors $\mathbf{e}_j$ to $\sigma_j$ and $-\sum_{i=1}^n \mathbf{e}_i$ to $\tau.$

   For $j = i+1,\dots, n,$ we let $\rho_j$ be the ray with primitive vector $\sum_{k = n-j+1}^n (k+j-n)\mathbf{e}_{k},$ so that $\rho_j$ is the ray introduced at the weighted blow-up $\Hilb^n(\mathbb{P}^1|0)_{\leq j}\to \Hilb^n(\mathbb{P}^1|0)_{\leq (j+1)}.$
\end{definition}

\begin{lemma}
    The Chow ring of $\Hilb^n(\mathbb{P}^1|0)_{\leq i})$ is given by $$\mathrm{CH}^\star(\Hilb^n(\mathbb{P}^1|0)_{\leq i})) = \frac{\mathbb{Z}[\sigma_1,\dots,\sigma_n, \tau, \rho_{i+1},\dots, \rho_{n}]}{\left\{\begin{array}{c}\sigma_{j} + \sum_{k = n-j+1}^{n}(k+j-n)\rho_k-\tau: j = 1,\dots, n, \\  \sigma_j\rho_{n-j}: j = 1,\dots, n-i+1, \\ \tau\rho_n, \sigma_{n-i}\sigma_{n-i+1}\dots\sigma_{n}\end{array}\right\}}$$
    In this isomorphism, each $\sigma_j,\tau\in \mathrm{CH}^1(\Hilb^n(\mathbb{P}^1|0)_{\leq i})$ is the proper transform of the coorindate hyperplane on $\Sym^n(\mathbb{P}^1),$ and each $\rho_j\in \mathrm{CH}^1(\Hilb^n(\mathbb{P}^1|0)_{\leq i})$ is the exceptional divisor introduced by the intermediate weighted blow-up $\Hilb^n(\mathbb{P}^1|0)_{\leq j}.$ 
\end{lemma}

\begin{remark}
    In the above, the proper transforms of $\tau$  and $\rho_j$ in the weighted blow-ups agree with their pullback.
\end{remark}

{We simplify the isomorphism as}$$\mathrm{CH}^\star(\Hilb^n(\mathbb{P}^1|0)_{\leq i})) = \frac{\mathbb{Z}[\tau, \rho_{i+1},\dots, \rho_{n}]}{\left\{\begin{array}{c} \left(\tau-\sum_{k = n-j+1}^{n}(k+j-n)\rho_k\right)\rho_{n-j}: j = 1,\dots, n-i+1, \\ \tau\rho_n, \prod_{j = n-i}^n (\tau-\sum_{k = n-j+1}^{n}(k+j-n)\rho_k) \end{array}\right\}}$$

\section{Logarithmic Hilbert schemes via weighted blow-ups}\label{sec:wbu}

\subsection{Relevant results for weighted blow-ups}
In this subsection we will assume the basic definitions of weighted blow-ups following the work of Quek--Rydh \cite{quek-rydh-weighted-blow-up}. A shorter introduction to weighted blow-ups can be found in \cite{arena2024criterion}.

Intuitively, weighted blow-ups can be thought of as birational transformations replacing a chosen center $\cZ$ by a divisor $\cE$, where $\cE \to \cZ$ has the structure of a bundle of weighted projective stacks $\cP(a_1,\dots,a_d)$. 

Choosing the (strictly positive integer) weights $a_1,\dots,a_d$ can be a subtle task, as it requires to assign $\Gm$-weights to local trivializations of the normal bundle of $\cZ$ in $\cX$ that are compatible across the open cover. In this paper we will restrict our attention to weighted blow-ups with extra regularity properties, namely what Quek--Rydh defined as regular weighted blow-ups.

\begin{remark}
    As a weighted blow-up is obtained as a quotient of the deformation to the weighted normal cone via the action of $\Gm,$ the weighted normal cone is at the heart of understanding the geometry and intersection theory of weighted blow-ups. In our setting, the weighted blow-up centers are regularly embedded, and we will study their weighted normal bundles both in this section (to prove the weighted blow-up construction) and in Section \ref{subsec: topChernClass}, where we compute its top Chern class to apply the weighted analogue of Keel's formula. It is important to remark that the weighted normal bundle is a twisted weighted vector bundle (using the definition of \cite{quek-rydh-weighted-blow-up}) and it is not necessarily isomorphic to the normal bundle.
\end{remark}

The main goal of this section is to prove the maps of Theorem \ref{thm:main} are, indeed, weighted blow-ups. Constructing weighted blow-ups from the base is a rather complicated endeavor. However, as long as the objects in questions are smooth Deligne-Mumford stacks, checking whether or not a given map is a weighted blow-up is a significantly easier matter, thanks to the following weighted blow-down criterion.

\begin{theorem}[\cite{arena2024criterion}, Theorem 1.1]\label{thm:blowdown} Let $\cY$ and $\cZ$ be smooth separated Deligne-Mumford stacks over a field of characteristic 0, and let $g: \cE \ra \cZ$ be a fibration of positive dimensional weighted projective stacks $\cP(a_1,\dots,a_n)$. Assume that there is a closed embedding $\cE \hookrightarrow \cY$ with $\cE$ a Cartier divisor in $\cY$, such that on every fiber $z \in \cZ,$ we have $\cO_{\cY}(\cE)|_{\cE_z} \cong \cO_{\cE_z}(-1)$.
Then there is a smooth, separated and tame Deligne-Mumford stack $\cX$, with two maps $\iota: \cZ \ra \cX$ and $f: \cY \ra \cX$ such that $\iota$ is a closed embedding and $f$ is a weighted blow-up with reduced center $\cZ$. Moreover, the resulting square 
$$\begin{tikzcd}
  \cE \arrow[hook]{r} \arrow{d}{g} & \cY \arrow{d}{f} \\
  \cZ \arrow[hook]{r}{\iota} & \cX
\end{tikzcd} $$ is a pushout in algebraic stacks.
\end{theorem}

In this paper we will leverage the weighted blow-up description to compute the Chow rings of the intermediate moduli spaces $\Hilb(C|D)_{\le \bi}$ by iteratively applying the weighted Keel's formula below.

\begin{remark}
    The original hypothesis of the theorem requires $\cE \to \cZ$ to be a weighted projective stack bundle such that $\cO_\cY(\cE)|_{\cE} \cong \cO_\cE(-1) \otimes g^*\cL$ with $\cL$ a line bundle over $\cZ$. We argue below why the two hypothesis are equivalent. 
    
    Let $\cL:=\cO_\cY(-\cE)|_{\cE},$ the construction $\Spec_{\cO_\cZ}(\bigoplus_{k\geq 0} \pi_*\cL^{\otimes k}):=\cN$ defines a twisted weighted vector bundle $\cN \to \cZ$ with the $\mathbb{G}_m$-action coming from the grading: this follows from a fiber-wise calculation of $R^j(\pi|_{\pi^{-1}(z)})_* ({\cL}|_{\pi^{-1}(z)})^{\otimes k})$ together with cohomology and base change theorem for Deligne--Mumford stacks \cite{Hall}.

    We are left to prove that $$\cE \cong [\Spec_{\cO_\cZ}(\bigoplus_{k\ge0} g_*\cL^{\otimes k}) \smallsetminus V(\bigoplus_{k>0} g_*\cL^{\otimes k})]/\Gm]=\cP(\cN).$$ The map $\pi: \cE \ra \cZ$ and $\cL$ satisfy the hypothesis of \cite[Section 1.7.1]{quek-rydh-weighted-blow-up}. Indeed for each $\varepsilon \in \cE$ there exists a positive integer $N$ such that $\pi^*\pi_*\cL^{\otimes N} \to \cL^{\otimes N}$ is surjective at $\varepsilon$, since it is true over the fibers $\cE_z=\pi^{-1}(z)$.
This implies that there is a map $\cE \to \cP(\cN)$ and the map is an isomorphism by Zariski's main theorem.

\end{remark}

\begin{theorem}\cite[Corollary 6.5]{Arena2025} \label{keel}
    Assume $i^*:\CHs(\cX)\to \CH^\star(\cZ)$ is surjective, then the integral Chow ring of the weighted blow-up $\cY\to \cX$ along $\cZ$ is given by
$$\CH^\star(\cY)\cong \frac{\CH^\star(\cX)[t]}{(t\cdot \ker(i^*), Q(t))},$$
where $Q(t)\in \CHs(\cX)[t]$ is any polynomial that pulls back to $c_{\mathrm{top}}^{\G_m}(\mathcal{N}_\cZ \cX)(t)\in \CHs(\cZ)[t]$ and such that $Q(0)=[\cZ] \in \CH^\star(\cX)$.
\end{theorem}

\subsection{Intermediate blow-ups}\label{sec:inducstep}

Let us start by constructing the sequence in Theorem \ref{thm:main}.

\begin{definition}\label{defn:EpsilonI}
    Fix $1\leq r\leq \ell$, let $\bi=(i_1,\dots,i_r,\dots, i_\ell)$, define $\bi-1_r:=\bi-\mathbf{e}_{r}=(i_1, \dots i_r-1, \dots, i_\ell).$ Let $\cE_{\bi}^{(r)} \subset \Hilb^n(C|D)_{\le \bi-1_j}$ be the closed substack parameterizing $(\tilde{C}, Z)$ where exactly $i_r$ points are supported on the bubble of $\tilde C$ containing $\tilde p_r$. Let $\cZ^{(r)}_{\bi} \subset \Hilb^n(C|D)_{\le \bi}$ be the closed substack where exactly $i_r$ points are supported on $\tilde{p}_r.$
\end{definition}

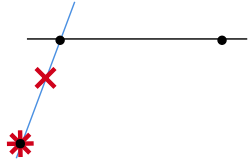
\begin{figure}[htp]
\resizebox{.2\textwidth}{!}{

\begin{tikzpicture}[x=0.75pt,y=0.75pt,yscale=-1,xscale=1]

\draw [color={rgb, 255:red, 74; green, 144; blue, 226 }  ,draw opacity=1 ]   (310,115) -- (288,174) ;
\draw  [color={rgb, 255:red, 208; green, 2; blue, 27 }  ,draw opacity=1 ][fill={rgb, 255:red, 208; green, 2; blue, 27 }  ,fill opacity=1 ] (301.97,139.92) -- (302.77,140.73) -- (299.8,143.7) -- (302.77,146.67) -- (301.97,147.47) -- (299,144.5) -- (296.03,147.47) -- (295.23,146.67) -- (298.2,143.7) -- (295.23,140.73) -- (296.03,139.92) -- (299,142.9) -- cycle ;
\draw [color={rgb, 255:red, 0; green, 0; blue, 0 }  ,draw opacity=1 ]   (292,129) -- (375,129) ;
\draw  [fill={rgb, 255:red, 0; green, 0; blue, 0 }  ,fill opacity=1 ] (367,129.5) .. controls (367,128.67) and (366.33,128) .. (365.5,128) .. controls (364.67,128) and (364,128.67) .. (364,129.5) .. controls (364,130.33) and (364.67,131) .. (365.5,131) .. controls (366.33,131) and (367,130.33) .. (367,129.5) -- cycle ;
\draw  [fill={rgb, 255:red, 0; green, 0; blue, 0 }  ,fill opacity=1 ] (306,129.5) .. controls (306,128.67) and (305.33,128) .. (304.5,128) .. controls (303.67,128) and (303,128.67) .. (303,129.5) .. controls (303,130.33) and (303.67,131) .. (304.5,131) .. controls (305.33,131) and (306,130.33) .. (306,129.5) -- cycle ;
\draw  [color={rgb, 255:red, 208; green, 2; blue, 27 }  ,draw opacity=1 ][fill={rgb, 255:red, 208; green, 2; blue, 27 }  ,fill opacity=1 ] (292.47,164.73) -- (293.27,165.53) -- (290.3,168.5) -- (293.27,171.47) -- (292.47,172.27) -- (289.5,169.3) -- (286.53,172.27) -- (285.73,171.47) -- (288.7,168.5) -- (285.73,165.53) -- (286.53,164.73) -- (289.5,167.7) -- cycle ;
\draw  [fill={rgb, 255:red, 0; green, 0; blue, 0 }  ,fill opacity=1 ] (291,168.5) .. controls (291,167.67) and (290.33,167) .. (289.5,167) .. controls (288.67,167) and (288,167.67) .. (288,168.5) .. controls (288,169.33) and (288.67,170) .. (289.5,170) .. controls (290.33,170) and (291,169.33) .. (291,168.5) -- cycle ;

\draw  [color={rgb, 255:red, 208; green, 2; blue, 27 }  ,draw opacity=1 ][fill={rgb, 255:red, 208; green, 2; blue, 27 }  ,fill opacity=1 ] (284.74,169.14) -- (284.72,168.01) -- (288.92,167.94) -- (288.86,163.74) -- (289.99,163.72) -- (290.06,167.92) -- (294.26,167.86) -- (294.28,168.99) -- (290.08,169.06) -- (290.14,173.26) -- (289.01,173.28) -- (288.94,169.08) -- cycle ;
\draw  [fill={rgb, 255:red, 0; green, 0; blue, 0 }  ,fill opacity=1 ] (288.42,167.46) .. controls (287.85,168.05) and (287.86,169) .. (288.46,169.58) .. controls (289.05,170.15) and (290,170.14) .. (290.58,169.54) .. controls (291.15,168.95) and (291.14,168) .. (290.54,167.42) .. controls (289.95,166.85) and (289,166.86) .. (288.42,167.46) -- cycle ;

\end{tikzpicture}}
\caption{A point in $\mathcal{Z}_{2}^{(0)}\cap \mathcal{E}_{3}^{(0)}\subset \Hilb^{3}(\PP^1|0+\infty)_{\leq (2,2)}.$ }
\label{fig:E3Z2}
\end{figure} 

\begin{remark}Notice that $\cE_{\bi}^{(r)}$ defines a codimension $1$ locus in $\Hilb^n(C|D)_{\le {{\bi}-1_r}}$, and it is dependent on the index $r$ via the difference $\bi- (\bi-1_r) = 1_r.$ Similarly, $\cZ^{(r)}_{\bi}$ has codimension $i_r.$
\end{remark}

\begin{remark}
    Let $\bi$ be any vector that has $r$-th entry equal to $i_r,$ and consider the morphism $$f_{i_r}:\Hilb^n(C|D)_{\leq \bi}\to \Hilb^n(C|p_r)_{\leq i_r}$$ that contracts bubbles from all markings but $p_r.$ Then the divisor $\calE_{\bi}^{(r)}$ on $\Hilb^n(C|D)_{\leq \bi}$ is pulled back from $\calE_{i_r}$ on $\Hilb^n(C|p_r)_{\leq i_r}.$ Therefore, we are justified in simplifying the notation $\calE_{\bi}^{(r)}$ to $\calE_{i_r}^{(r)}$ as a divisor on $\Hilb^n(C|D)_{\leq \bi}.$

    Let $\bi\leq \bi'$ with $r$-th entries $i_r, i_r'$ respectively such that $i_r< i_r'.$ Then the pullback of the divisor $\cE_{i_r'}^{(r)}$ along  $\Hilb^n(C|D)_{\leq \bi}\to \Hilb^n(C|D)_{\leq \bi'}$ is transverse to $\cE_{i_r}^{(r)}.$ Similarly, for any $r_1\neq r_2,$ and any $j_1,j_2$ the pullback of the divisor $\calE_{j_1}^{(r)}$ is transverse to the pullback of $\calE_{j_2}^{(r')}$ on $\Hilb^n(C|D)_{\leq \bi},$ whenever both are defined.
\end{remark}

\begin{customthm}{B}\label{thm:D=p}
  The natural map $$ \mathrm{Hilb}^n(C|D)_{\leq \bi-1_r}\to \mathrm{Hilb}^n(C|D)_{\leq \bi}$$ is a weighted blow-down with exceptional divisor $\cE_{i_r}^{(r)}$, center $\cZ^{(r)}_{i_r}$ and weights $(1,\dots,i_r).$
\end{customthm}

We prove Theorem \ref{thm:D=p} by verifying the criterion for weighted blow-downs in Theorem \ref{thm:blowdown}. Lemmas \ref{lem:TWPB}, \ref{lem:Cart}, and \ref{lem:O(-1)} will provide all the necessary ingredients.

{To simplify the notation, in Lemmas \ref{lem:TWPB}, \ref{lem:Cart}, and \ref{lem:O(-1)} we will use $\cZ,\ \cE,\ \cX,\ \cY$ to denote the spaces $\cZ^{(r)}_{\bi}, \ \cE^{(r)}_{\bi}, \ \Hilb^n(C|D)_{\bi},\ \Hilb^n(C|D)_{\le \bi-1_r}$ respectively.} {We will also suppress the dependency on the index of the marking $r$ and use $i$ to denote $i_r$ whenever there is no risk of ambiguity.}

\begin{lemma}\label{lem:TWPB}
    With the notation from Definition \ref{defn:EpsilonI}, for each geometric point $z\in \cZri,$ the fiber of $\piri:\cEri \to \cZri$ is a weighted projective stack $\cP(1,\dots, i).$
\end{lemma}

\begin{proof}
    Let $z\in \cZri,$ let $C_z$ be the expanded curve over $z,$ and let $\tilde{p}_i\in C_z$ be the endpoint along $p_i.$ The fiber $\piri^{-1}(z)$ is the data of attaching $i$ unordered points on the $\PP^1$-bubble attached to $\tilde{p}_i.$ Therefore, we have the isomorphism $$\piri^{-1}(z)\cong [(\mathrm{Sym}^i(\mathbb{A}^1)\setminus 0)/\mathbb{G}_m],$$ where $\mathbb{G}_m$ acts on $\mathrm{Sym}^i(\mathbb{A}^1)$ with weights $(1,\dots, i)$ specified by monomial functions on $\mathbb{A}^1.$ Thus, $$\piri^{-1}(z)\cong \mathcal{P}(1,\dots, i).$$
    \end{proof}

\begin{lemma}\label{lem:Cart}
$\cEri$ is a smooth Cartier divisor of $\Hilb^n(C|p)_{\leq \bi -1_r}.$
\end{lemma}
\begin{proof}
    The map $\cEri\to \Exp(C|D)$ is smooth for the same argument as in Corollary 3.11, in which we show that $\Hilb^n(C|D)_{\leq \bi}$ smooth.
\end{proof}

\begin{lemma} \label{lem:O(-1)}
    For all geometric point $z\in \cZri,$ we have $$\cO_{\Himor}(\cEri)|_{\piri^{-1}(z)} \cong \cO_{\piri^{-1}(z)}(-1).$$ 
\end{lemma}
\begin{proof}

We start by lifting the problem to the Hassett space as follows.

Associated to $\bi-1_r$ and $\bi$ are respectively the weights 
$$\cA=(1-i_1\delta_n, \dots, 1-(i_r-1)\delta_n, \dots, 1-i_\ell\delta_n ,\underbrace{\delta_n, \delta_n, \dots, \delta_n}_{n \text{ times}})$$ and $$\cB=(1-i_1\delta_n, \dots, 1-i_r\delta_n, \dots, 1-i_\ell\delta_n ,\underbrace{\delta_n, \delta_n, \dots, \delta_n}_{n \text{ times}})$$ from Definition \ref{defn:wtAbi}. 

By \cite[Proposition 4.5]{hassett}, the map $\rho_{\cA, \cB}: \Mbar_{g, \cA}\to \Mbar_{g, \cB}$ is an isomorphisms outside the boundary divisors of the form $\Mbar_{0,\cA_{I_r}} \times \Mbar_{g,\cA_J}$ with $$\cA_{I_r}=(1-(i_r-1)\delta_n, \underbrace{\delta_n, \dots, \delta_n }_{i_r \text{ times} }, 1) \ \text{ and } \ \cA_J=(1-i_1\delta_n, \dots, \underbrace{1}_{r\text{-th position}}, \dots, 1-i_\ell\delta_n,  \underbrace{\delta_n, \dots, \delta_n}_{n-i_r \text{ times} })),$$
where $\rho_{\cA, \cB}$ contracts the rational subcurves parametrized by $\Mbar_{0,\cA_I}$.  Analyzing the weights $\mathcal{A}_{I_r},$ we see that $$\Mbar_{0, \mathcal{A}_{I_r}}\cong [((\mathbb{A}^1)^i\setminus 0)/\mathbb{G}_m] \cong \mathbb{P}^{i_r-1},$$ and that the natural map $\Mbar_{0, \cA_{I_r}} = [((\mathbb{A}^1)^i\setminus 0)/\mathbb{G}_m]\to B\mathbb{G}_m$ agrees with the forgetful map $\Mbar_{0, \cA_{I_r}}\to \mathcal{M}_{0,2}\cong B\mathbb{G}_m,$ where the marking with weight 1 is identified with $\infty\in \PP^1,$ and the one with weight $1-(i_r-1)\delta_n$ is identified with $0\in \PP^1.$ Let $\Mbar_{g,\cB}^{(i_r)}\subset \Mbar_{g,\cB}$ be the (disconnected) closed substack of $\Mbar_{g, \cB}$ cut out by the condition that any $i_r$ of the $n$ markings with weight $\delta_n$ collide with the marking with weight $1-i_r\delta_n.$ We observe that each connected component of $\Mbar_{g,\cB}^{(i_r)}$ is isomorphic to $\Mbar_{g, \cA_J}.$ Then it is known \cite[Remark 4.6]{hassett} that the morphism $\rho_{\cA, \cB}$ is the blow-up of $\Mbar_{g,\cB}$ with center $\Mbar_{g,\cB}^{(i_r)} ,$ and the exceptional divisor of this blow-up is $$\tilde{\cE}^{(i_r)}_{g, \cA} := \bigsqcup_{I_r\subset [n]: |I_r| = i_r} \Mbar_{0, \cA_{I_r}}\times \Mbar_{g, \cA_J}\cong \Mbar_{g, \cB}^{(i_r)}\times \PP^{i_r-1}.$$ Therefore, the normal bundle of the exceptional divisor $\cO(\cE^{(i_r)}_{g, \cA})|_{\cE^{(i_r)}_{g, \cA}}$ agrees with the pullback of $\cO_{\PP^{i_r-1}}(-1),$ which is in turn the pullback of $\mathcal{O}(-1)$ along the map $\Mbar_{0, \cA_{I_r}}\to \mathcal{M}_{0,2}\cong B\mathbb{G}_m.$

Passing to the setting of logarithmic Hilbert schemes, we recall the forgetful maps $\phi_\cA: \Mbar_{g, \cA}\to \Mbar_{g,\ell},$ $\phi_\cB: \Mbar_{g, \cB}\to \Mbar_{g,\ell}$ and use the notation $$\tilde \cY=\phi_{\mathcal{A}}^{-1}(C, p_1,\dots,p_{\ell}) \ \text{and} \  \tilde{\cX} = \phi_{\mathcal{B}}^{-1}(C, p_1,\dots,p_{\ell})$$ then by Proposition \ref{prop:quotient}, we have relative coarse moduli space morphisms $[{\tilde \cY}/S_{n}]\to \Himor$ and $[{\tilde \cX}/S_{n}]\to \Hi$ over $\Exp(C|D).$ 
Let $\gamma: \tilde{\cY}\to \Himor$ be the composition of the quotient map and the relative coarse moduli space morphism.

Let $z\in \cZri,$ and pick any pre-image $\tilde{z}\in \Mbar_{g,\cB}^{(i_r)}\subset \Mbar_{g, \cB}$ along the $S_n$-quotient map. There is a commutative diagram \[\begin{tikzcd}
	{\rho_{\cA, \cB}^{-1}(\tilde{z})\cong \PP^{i_r-1}} & \\
	{\piri^{-1}(z)\cong \mathcal{P}(1,\dots, i_r)} & {\mathcal{M}_{0,2}\cong B\mathbb{G}_m}
	\arrow["\gamma|_{\rho_{\cA,\cB}^{-1}(\tilde{z})}"', from=1-1, to=2-1]
	\arrow["\alpha", from=1-1, to=2-2]
	\arrow["\beta"', from=2-1, to=2-2]
\end{tikzcd}\]

Because $\tilde{\cE}^{(i_r)}_{g, \cA}\cap \tilde{\cY}\subset \tilde{\cY}$ is pulled back from $\cEri\subset \Himor,$ we have $$\cO_{\tilde{\cY}}(\tilde{\cE}^{(i_r)}_{g, \cA}\cap \tilde{\cY}) = \gamma^* \cO_{\Himor}(\cEri).$$ Restricting to $\rho_{\cA,\cB}^{-1}(\tilde{z}),$ this becomes $$\cO_{\PP^{i_r-1}}(-1)\cong \gamma^*\cO_{\Himor}(\cEri)|_{(\piri)^{-1}(z)}.$$

On the other hand, we know that $$\cO_{\Himor}(\cEri)|_{(\piri)^{-1}(z)} = \beta^* \cO(d) = \cO_{\cP(1,\dots, i_r)}(d)$$ for some $d\in \mathbb{Z}.$ Comparing degrees yields $d = -1$ as desired.

\end{proof}


\begin{proof}[Proof of Theorem \ref{thm:D=p}]
We verify that the data of $\cE^{(r)}_{\bi}\subset \Hilb^n(C|D)_{\le \bi-1_r}\to \Hilb^n(C|D)_{\le \bi}$ and $\cE_{\bi}^{(r)}\to \cZ^{(r)}_{\bi}$ satisfies the hypotheses of Theorem \ref{thm:blowdown}. Smoothness of the stacks and the divisor have been proven in Lemmas \ref{cor:Hilbsm} and \ref{lem:Cart}. Lemma \ref{lem:O(-1)} verifies that $\cE^{(r)}_{\bi}\to \cZ^{(r)}_{\bi}$ is a twisted weighted projective bundle with weights $(1,\dots, i_r)$ and proves the desired statement about normal bundle of the divisor. Theorem \ref{thm:blowdown} then implies that there exists a unique space contracting $\cE$ to $\cZ$ that fits into the diagram. This implies the weighted blow-down must be isomorphic to $\Hilb^n(C|D)_{\le \bi}$.
\end{proof}
\subsection{Main theorem}

\begin{customthm}{A} \label{thm:main}
    The map $\Hilb^n(C|D)\to \Hilb^n(C) = \Sym^n(C)$ factors through a sequence of $n\cdot \ell$ weighted blow-ups $$\Hilb^n(C|D)=\Hilb^n(C|D)_{\le \mathbf{0}} \to \cdots  \to \Hilb^n(C|D)_{\leq \bi}  \to \cdots \to \Hilb^n(C|D)_{\le \boldsymbol{n}}=\Hilb^n(C)$$ where $\bi\in \mathbb{N}^\ell$ satisfies that $\mathbf{0}\leq \bi\leq \boldsymbol{n} = (n,\dots,n)$ entry-wise. Each $\Hilb^n(C|D)_{\leq \bi}$ is an intermediate moduli space of subschemes on expanded degenerations of $C$ along $D$ and each weighted blow-up is of the form \begin{equation} \label{eq: onestep}
       \Hilb^n(C|D)_{\le (i_1,\dots, i_j-1,\dots , i_\ell)} \ra \Hilb^n(C|D)_{\le (i_1, \dots ,i_\ell)}.
    \end{equation}
     with weights $1, \dots , i_j$, for some $0 \le j \le \ell$.
\end{customthm}

\begin{proof}
    Theorem \ref{thm:D=p} ensures that each of the maps described above is a weighted blow-up with the desired center and weights. Furthermore, given the two vectors $(i_1,\dots, i_j-1,\dots , i_k-1, \dots i_\ell)$ and $(i_1, \dots ,i_\ell)$ with $j \lneq k$, the composition of two weighted blow-ups of the form (\ref{eq: onestep}) $$\Hilb^n(C|D)_{\le (i_1,\dots, i_j-1,\dots , i_k-1, \dots i_\ell)} \ra \Hilb^n(C|D)_{\le (i_1, \dots ,i_\ell)} $$ is independent of the order with which the blow-ups are performed. This follows from the analogous proposition on Hassett spaces proved in \cite[Proposition 4.4]{hassett}.
\end{proof}

\begin{customcor}{C}\label{cor:P1}
    The logarithmic Hilbert stacks $\Hilb^n(\PP^1|0)$, $\Hilb^n(\PP^1|\infty)$ and $\Hilb^n(\PP^1|0+\infty)$ are toric stacks.
\end{customcor}

\begin{proof}
    From Theorem \ref{thm:D=p}, both logarithmic Hilbert stacks are iterated weighted blow-ups of $\Hilb^n(\PP^1),$ and we observe that each blow-up center is the proper transforms of some toric substack of $\PP^n\cong\Hilb^n(\PP^1).$ In particular, the centers are toric substacks themselves, so the weighted blow-ups are all toric.
\end{proof}

\begin{remark}
Continuing the discussion set out in §\ref{sec:2P10}, we note that the geometry of logarithmic Hilbert schemes of $\PP^1$ is particularly accessible via homogeneous polynomials on $\PP^1$: the blow-up centers are proper transforms of their coordinate subspaces. As a further illustration, we include (Figure \ref{fig:Hilb3P1}) the toric fan picture of $\Hilb^3(\PP^1|0).$
\end{remark}
\input{Hilb3P1}

\begin{corollary}
    The intermediate moduli spaces $\Hilb^n(C|D)_{\le \bi},$ in particular the logarithmic Hilbert stack $\Hilb^n(C|D),$ have projective coarse moduli spaces.
\end{corollary}
\begin{proof}
    This is because a weighted blow-up induce a projective morphism on coarse moduli spaces, and $\Hilb^n(C)\cong \Sym^n(C)$ is projective.
\end{proof}

\section{Integral Chow ring of logarithmic Hilbert schemes} 
In this section, we give a presentation of the integral Chow ring of $\Hilb^n(C|D)_{\le \bi}$ in terms of $\mathrm{Sym}^n(C)$ and exceptional divisors $\{\epsilon_{j}^{(r)}\}_{i_r+1\leq j\leq n}$ (Theorem \ref{thm:chowringD}). The calculation involves determining the normal bundle of the weighted blow-up centers and applying Keel's formula \ref{keel} to the iterated weighted blow-ups $\Hilb^n(C|D)_{\le \bi}\to \mathrm{Sym}^n(C)$. We will focus on the case of $D = p,$ which generalizes in a straightforward way.

\subsection{The equivariant top Chern class of $\boldsymbol{\cN}_{\boldsymbol{\cZ}_i}\textbf{Hilb}^n\boldsymbol{(C|p)}_{\le i}$} \label{subsec: topChernClass}

The key ingredient for the Chow ring computation, is to understand the $\mathbb{G}_m$-equivariant top Chern class of the normal bundle $\mathcal{N}_{\cZ_{\bi}}\Hilb^n(C|p)_{\leq \bi}$. We will compute it by introducing a nested sequence of hypersurfaces in the intermediate moduli spaces.

\begin{definition}
Let $H_{n,i}\subset \Hilb^n(C|p)_{\leq i}$ be the hypersurface specified by one point lying on $\tilde{p}.$ As stacks, $H_{n,i}\cong \Hilb^{n-1}(C|p)_{\leq i-1}.$
\end{definition}
As $H_{n,i}\subset\Hilb^n(C|p)_{\leq i}$ is a Cartier divisor, the normal bundle $\mathcal{N}_{H_{n,i}}\Hilb^n(C|p)_{\leq i} $ agrees with $\mathcal{O}(H_{n,i})|_{H_{n,i}}.$ We start by recalling the following fact about the base case $H_{n,n}\subset \Sym^{n}(C).$

\begin{lemma}\cite[Remark 1.2]{Moonen_Polishchuk} \label{lem:LonHnn}
  There is a well-defined line bundle $\mathcal{L}$ on $H_{n,n}\cong \Sym^{n-1}(C)$ that is isomorphic to:\begin{enumerate}
    \item the normal bundle $\mathcal{N}_{H_{n,n}}\Sym^n(C),$
    \item the line bundle $\mathcal{O}_{\Sym^{n-1}(C)}(H_{n-1,n-1})$ associated to the divisor $H_{n-1,n-1}\subset \Sym^{n-1}(C),$
    \item restriction of $\mathcal{N}_{H_{n+k,n+k}}\Sym^{n+k}(C)$ and $\mathcal{O}_{\Sym^{n+l}(C)}(H_{n+l,n+l})$ for any $k\geq 0, l\geq -1.$
  \end{enumerate}
\end{lemma}

\begin{remark}
    When $n>2g-2,$ \cite[§1]{Moonen_Polishchuk} also identifies the line bundle $\cL$ with $\cO_{\Sym^n(C)}(1)$, where $\Sym^n(C)$ has the structure of a projective bundle over the Jacobian. Here, the degree-$n$ Picard group $\Pic^n(C)$ is implicitly identified with the Jacobian $\mathrm{Jac}(C)$ by $D\mapsto D-np$.
\end{remark}

Now we turn to the normal bundles of $H_{n,i}\subset \Hilb^n(C|p)_{\leq i}$ and relate them to $\mathcal{L}.$ In the remainder of the section, we refine the notations of $\cE_i, \epsilon_i,$ and $\cZ_i$, since it will be necessary to specify which ambient spaces the classes live in.

\begin{definition}\label{def: doubleindexes}
     Let $\cE_{n,i}$ and $\epsilon_{n,i}=[\cE_{n,i}]\in \mathrm{CH}^1(\Hilb^n(C|p)_{\leq i-1})$ denote the (class of) exceptional divisor introduced along the weighted blow-up $\Hilb^n(C|p)_{\leq i-1}\to \Hilb^n(C|p)_{\leq i}.$ They specify the divisor of subschemes which has length-$i$ supported on the last bubble. 
    
    In general, for $k \geq i,$ let $\cE_{n,k}$ and $\epsilon_{n,k}\in \mathrm{CH}^1(\Hilb^n(C|p)_{\leq i})$ be the proper transforms of $\cE_{n,k}$ and $\epsilon_{n,k}$ on $\Hilb^n(C|p)_{\leq k},$ which agree with the pullback along the iterated weighted blow-ups $\Hilb^n(C|p)_{\leq i}\to \Hilb^n(C|p)_{\leq k}.$

    Let $\mathcal{Z}_{n,i}\subset \Hilb^{n}(C|p)_{\leq i}$ denote the closed substack where exactly $i$ points lie on the endpoint.
\end{definition}

\begin{lemma}\label{lem:blwoupnormal}
  Let $f: \Hilb^n(C|p)_{\leq i-1}\to \Hilb^n(C|p)_{\leq i}$ be the weighted blow-up, then 
$${c_1(\mathcal{O}_{\Hilb^n(C|p)_{\leq i-1}}(H_{n,i-1})) = f^* c_1(\mathcal{O}_{\Hilb^n(C|p)_{\leq i}}(H_{n,{i}}))-i\epsilon_{n,i}}.$$
  
  Restricting to the proper transform morphism $f: H_{n,i-1}\to H_{n,i},$ we have 
  \begin{equation}\label{eqn:NHni}{c_1(\mathcal{N}_{H_{n,i-1}} \Hilb^n(C|p)_{\leq i-1}) = f^* c_1(\mathcal{N}_{H_{n,i}})-i\epsilon_{ n-1, i-1}, \text{where }\epsilon_{n,i}|_{H_{n,i-1}} = \epsilon_{n-1, i-1}.}\end{equation}

\end{lemma}

\begin{proof}
    
    Because the line bundles $\mathcal{O}_{\Hilb^n(C|p)_{\leq i-1}}(H_{n,i-1})$ and $f^* \mathcal{O}_{\Hilb^n(C|p)_{\leq i}}(H_{n,{i}})$ agree on $\Hilb^n(C|p)_{\leq (i-1)}\setminus \cE_{n, i-1},$ the excision sequence associated to the open immersion $\Hilb^n(C|p)_{\leq i-1}\setminus \cE_{n, i}\subset \Hilb^n(C|p)_{\leq i-1}$ implies that $c_1(\mathcal{O}_{\Hilb^n(C|p)_{\leq i-1}}(H_{n,i-1})) = f^* c_1(\mathcal{O}_{\Hilb^n(C|p)_{\leq i}}(H_{n,{i}}))+m\epsilon_{n,i}$ for some $m\in \mathbb{Z}.$ The fact that $m = -i$ now follows from the weighted blow-up formula in rational Chow group \cite{ArenaThesis}.
    
\end{proof}

\begin{remark}
  This agrees with the Stanley--Reisner ring calculation on $\Hilb^n(\mathbb{P}^1|0)_{\leq i}$ from §\ref{subsec:SR}, in which we have $\rho_1 + n \epsilon_n = \tau,$ with $\tau$ being the pullback of the normal bundle.
\end{remark}


Consider the filtration of closed substacks $$\mathcal{Z}_{n,i} =\Hilb^{n-i}(C|p)_{\le 0}  = H_{n-i+1, 1}\subset \Hilb^{n-i+1}(C|p)_{\le{1}}=H_{n-i+2, 2} \subset \cdots$$ $$\cdots \subset  \Hilb^{n-1}(C|p)_{\le i-1}= H_{n,i}\subset \Hilb^n(C|p)_{\leq i}.$$ Applying Lemma \ref{lem:blwoupnormal} to each step of the filtration $H_{n-k, i-k}\subset \Hilb^{n-k}(C|p)_{\leq i-k},$ we have the following formula for $c_{\mathrm{top}}^{\mathbb{G}_m} (\cN_{\cZ_{n,i}}\Hilb^n(C|p)_{\leq i}).$

\begin{corollary}
  The top Chern class of $\mathbb{G}_m$-equivariant normal bundle of $\mathcal{Z}_{i}\subset \mathcal{X}_i$ is given by 
  $$c_{\mathrm{top}}^{\mathbb{G}_m} (\cN_{\cZ_{n,i}}\Hilb^n(C|p)_{\leq i}) = \prod_{k=1}^i (k\cdot t + c_1(\mathcal{L})- \sum_{j = i-k+1}^{n}j\epsilon_{n-i, j-i+k}).$$

\end{corollary}
\begin{proof}
  We apply splitting principle to get the product $$\underbrace{(i\cdot t + c_1(\mathcal{N}_{H_{n,i}}\mathcal{X}_i))}_{H_{n,i}\subset \mathcal{X}_i} \underbrace{((i-1)t + c_1(\mathcal{N}_{H_{n-1,i-1}}\Hilb^{n-1}(C|p)_{\leq i-1}))}_{H_{n-1,i-1}\subset H_{n,i}}\cdots \underbrace{(t + c_1(\mathcal{N}_{H_{n-i+1, 1}}\Hilb^{n-i+1}(C|p)_{\leq 1}))}_{H_{n-i+1,1}\subset H_{n-i+2,2}}.$$ Here $H_{n,i}\subset \mathcal{X}_i$ has weight $i$ because asking for one point to be on $\tilde{p}$ is a homogeneous degree $i$ equation on local chart. Now we iterate formula (\ref{eqn:NHni}) to get the product displayed above. 
\end{proof} 

We encode the Chern class calculations in terms of the following polynomials.

\begin{definition} \label{defn:Qmh}We define the polynomials $Q_{m,h}(t_m,\dots,t_h) \in \CHs(\Sym^n(C))[t_m,\dots,t_h]$ such that
  $$Q_{m,h}(\epsilon_{m-h,m-h},\dots, \epsilon_{m-h,1}, t) = c_{\mathrm{top}}^{\mathbb{G}_m}(\mathcal{N}_{\mathcal{Z}_{m,h}}\Hilb^m(C|p)_{\leq h}) \in  \mathrm{CH}^h(\cZ_{m,h})[t]$$ where we use the identification $\cZ_{m,h}\cong \Hilb^{m-h}(C|p)_{\leq 0}.$ 
\end{definition}

Precisely, when $m>h>0,$ $$Q_{m,h}(t_m, \dots, t_{h+1}, t_h):=\prod_{k=1}^h (kt_h + c_1(\mathcal{L})-\sum_{j=1}^{m-h}(k+j)t_{j+h}).$$

When $m = h,$ $$Q_{m,m}(t_h):=\prod_{k=1}^h (kt_h + c_1(\mathcal{L})).$$ We formally define $Q_{0, k} = 0$ for all $k.$

\subsection{The case of $\textbf{ \text{Hilb}}^{\boldsymbol{n}}{\boldsymbol{(C|p)}}_{\le \boldsymbol{i}}$}
\begin{definition}\label{def:symimmersiononepoint}
    Let $p\in C.$ For $j\geq i,$ define $s_{ij}: \Sym^i(C) \ra \Sym^j(C)$ as the map given by $Z\mapsto Z + (j-i)p.$
\end{definition}

\begin{proposition}
The restriction map $\iota_{i,n}^*:\mathrm{CH}^\star(\Hilb^{n}(C|p)_{\leq i})\to \mathrm{CH}^\star(\cZ_{n,i})$ is surjective.
\end{proposition}
\begin{proof}
    We prove the statement by induction on $2n-i$. If $i=n$, then the restriction map is $$\CHs (\Sym^n(C)) \ra \CHs (\Sym^0(C))=\CHs (*)=\bZ$$ which is clearly surjective. 
    Let us now assume $\iota_{i',n'}$ is surjective for $2n'-i'< 2n-i$ and prove the statement for $2n-i$.
     By Theorem \ref{thm:main}, $\Hilb^n(C|p)_{\le i}$ is a weighted blow-up of $\Hilb^n(C|p)_{\le i+1}.$ By inductive hypothesis on $(n',i') = (n,i+1),$ we can apply Keel's formula (Theorem \ref{keel}), and deduce the Chow ring of $\Hilb^{n}(C|p)_{\leq i}$ is generated by $\mathrm{CH}^\star(\mathrm{Sym}^n(C))$ and $\epsilon_{n,n}, \dots, \epsilon_{n, i+1}.$ Via the isomorphism $\cZ_{n,i}\cong \Hilb^{n-i}(C|p)_{\leq 0},$ and a similar reasoning, the Chow ring of $\cZ_{n,i}$ is generated by $\mathrm{CH}^\star(\mathrm{Sym}^{n-i}(C))$ and $\epsilon_{n-i,n-i}, \dots, \epsilon_{n-i, 1}.$ We observe that the restriction map $\iota_{i,n}^*$ sends $\epsilon_{n,j}$ to $\epsilon_{n-i, j-i}.$ The commutative diagram \[\begin{tikzcd}
	{\cZ_{n,i}} & {\Hilb^{n}(C|p)_{\leq i}} \\
	{\mathrm{Sym}^{n-i}(C)} & {\mathrm{Sym}^{n}(C)}
	\arrow[hook, from=1-1, to=1-2, "\iota_{i,n}"]
	\arrow[from=1-1, to=2-1]
	\arrow[from=1-2, to=2-2]
	\arrow[hook, from=2-1, to=2-2, "s_{(n-i),n}"]
    \end{tikzcd}\] implies that the pullback $\iota_{i,n}^*: \mathrm{CH}^\star(\Hilb^{n}(C|p)_{\leq i})\to \CHs(\cZ_{n,i})$ extends $$s_{(n-i),n}^*: \mathrm{CH}^\star(\mathrm{Sym}^{n}(C))\to \mathrm{CH}^\star(\mathrm{Sym}^{n-i}(C)).$$ Therefore, to prove that the map $\iota_{i,n}^*$ is surjective, it suffices to check that each $s_{(n-i),n}^*$ is surjective. The work of Kimura--Vistoli \cite[§1.8]{KV96} constructed correspondences $\delta_n: \mathrm{Sym}^n(C)\vdash \mathrm{Sym}^{n-1}(C)$ such that $\delta_n\circ s_{(n-1)n} = \mathrm{id},$ so that $s_{(n-1)n}^*: \mathrm{CH}^\star(\mathrm{Sym}^{n}(C))\to \mathrm{CH}^\star(\mathrm{Sym}^{n-1}(C))$ is surjective: see also \cite[Remark 5.4]{Yin2016}. Therefore, the generators of $\mathrm{CH}^\star(\Hilb^{n}(C|p)_{\leq i})$ surjects onto the generators of $\mathrm{CH}^\star(\cZ_{n,i}).$ 
\end{proof}


\begin{theorem}
    
\label{thm:chowringp}
  The Chow ring $\mathrm{CH}^\star(\Hilb^n(C|p)_{\leq i})$, has a presentation 

  $$\mathrm{CH}^\star(\Hilb^n(C|p)_{\leq i}) = \frac{\mathrm{CH}^\star(\mathrm{Sym}^n(C))[\epsilon_{n},\dots, \epsilon_{i+1}]}{\left\{\begin{array}{c}
   \text{for each }j = n, \dots, i+1: Q_{n,j}(\epsilon_{j})(\epsilon_{n},\dots,\epsilon_{j+1}); \ \ker(s_{(n-j)n}^*) \cdot \epsilon_{j}; \\ \text{for each } k = n-j, \dots, 1: Q_{n-j,k}(\epsilon_{n},\dots, \epsilon_{j+k+1}, \epsilon_{j+k})\cdot \epsilon_{n-j} 
  \end{array}\right\}}$$ where $\epsilon_i=[\cE_i]$ is the class of the exceptional divisor of the $(n-i+1)$-th weighted blow-up and corresponds to the locus parameterizing $i$ points supported on the last bubble.  
\end{theorem}
\begin{example}
  When $i = n-1,$ the indices in the list are $j = n.$ Since $n-j = 0,$ all the $Q_{n-j,k}=0,$ so we get $$\mathrm{CH}^\star(\mathrm{Sym}^n(C))[\epsilon_{n,n}]/(Q_{n,n}(\epsilon_{n,n}),\ker (s_{0n}^*)\epsilon_{n,n}).$$
\end{example}
\begin{proof}

For the purpose of the proof, let us rewrite the presentation with the double indexes as in Definition \ref{def: doubleindexes}, so that we can more easily apply induction. With the refined notation, the desired formula reads 

$$\mathrm{CH}^\star(\Hilb^n(C|p)_{\leq i}) = \frac{\mathrm{CH}^\star(\mathrm{Sym}^n(C))[\epsilon_{n,n},\dots, \epsilon_{n,i+1}]}{\left\{\begin{array}{c}
  \text{for each } j = n, \dots, i+1: Q_{n,j}(\epsilon_{n,n},\dots,\epsilon_{n,j+1}, \epsilon_{n,j}); \ \ker(s_{(n-j)n}^*) \cdot \epsilon_{n,j};\\ \text{for each } k = n-j, \dots, 1: Q_{n-j,k}(\epsilon_{n,n},\dots, \epsilon_{n,j+k+1}, \epsilon_{n,j+k})\cdot \epsilon_{n,n-j},  
  \end{array}\right\}}.$$

  We use lexicographic ordering on $(n, n-i)$ and run induction on this totally ordered set. When $n-i = 0,$ the indexing set of $j$ runs from $n, n+1$ which we treat as the empty set, hence the inductive hypothesis on the row $\{(n,0)\}$ holds.

  We present $\Hilb^n(C|p)_{\leq i},$ as blow-up of $\Hilb^n(C|p)_{\leq (i+1)}$ which, by inductive hypothesis, has a presentation $$\CHs(\Hilb^n(C|p)_{\leq (i+1)}) = \frac{\mathrm{CH}^\star(\mathrm{Sym}^n(C))[\epsilon_{n,n},\dots, \epsilon_{n,i+2}]}{\left\{\begin{array}{c}
  \text{for each } j = n, \dots, i+2: Q_{n,j}(\epsilon_{n,n},\dots,\epsilon_{n,j+1}, \epsilon_{n,j}); \  \ker(s_{(n-j)n}^*) \cdot \epsilon_{n,j};\\\text{for each } k = n-j, \dots, 1: Q_{n-j,k}(\epsilon_{n,n},\dots, \epsilon_{n,j+k+1}, \epsilon_{n,j+k})\cdot \epsilon_{n,n-j}, 
  \end{array}\right\}}$$ 

  Applying Keel's formula, we get \[\mathrm{CH}^\star(\Hilb^n(C|p)_{\leq i}) = \frac{\mathrm{CH}^\star(\Hilb^n(C|p)_{\leq i+1})}{(Q_{n,i+1}(\epsilon_{n,i+1}); \ \ker \iota^*_{i+1,n} \cdot \epsilon_{n,i+1})},\] where $\iota_{i+1,n}: \mathcal{Z}_{n,i+1}\to \Hilb^{n}(C|p)_{\leq i+1}$ is the closed embedding. To perform the inductive step, we determine $\ker \iota^*_{i+1,n}.$
  
  In the following, let $m:=n-(i+1).$ By inductive hypothesis, the substack $\mathcal{Z}_{n,i+1}\cong \Hilb^{m}(C|p)_{\leq 0}$ has presentation as $$\CHs(\Hilb^m(C|p)_{\le 0})=\frac{\mathrm{CH}^\star(\mathrm{Sym}^{m}(C))[\epsilon_{m,m},\dots, \epsilon_{m,1}]}{\left\{\begin{array}{c}
   \text{for each } j = m, \dots, 1: Q_{m,j}(\epsilon_{m,m},\dots,\epsilon_{m,j+1}, \epsilon_{m,j}); \  \ker(s_{(m-j),m}^*) \cdot \epsilon_{m,j};\\\text{for each } k = m-j, \dots, 1: Q_{m-j,k}(\epsilon_{m,m},\dots, \epsilon_{m,j+k+1}, \epsilon_{m,j+k})\epsilon_{m,m-j}
  \end{array}\right\}},$$ and the map $$\iota_{m,n}^*: \mathrm{CH}^\star(\Hilb^n(C|p)_{\leq (i+1)})\to \mathrm{CH}^\star(\Hilb^{m}(C|p)_{\leq 0})$$ extends $$s_{m,n}^*: \mathrm{CH}^\star(\mathrm{Sym}^n(C))\to \mathrm{CH}^\star(\mathrm{Sym}^{m}(C))$$ by sending $\epsilon_{n,j}\mapsto \epsilon_{m, j-(i+1)}.$ By Lemma \ref{lem:rings} below, the kernel of the map is generated by $\ker s_{m,n}^*$ and lifts of the relations in the presentation of $\mathrm{CH}^\star(\Hilb^{m}(C|p)_{\leq 0})$ given above.
  
  We observe that for each $j= m, \dots, 1,$ the restriction map $\iota_{m,n}^*$ sends $$Q_{m, j}(\epsilon_{n, j+(i+1)})(\epsilon_{n,n},\dots, \epsilon_{n, j+i+2})\mapsto Q_{m,j}(\epsilon_{m,j})(\epsilon_{m,m},\dots,\epsilon_{m,j+1}),$$ {$$\ker(s_{(\tilde{n}-j)n}^*)\epsilon_{n,j+(i+1)}\mapsto \ker(s_{(\tilde{n}-j)\tilde{n}}^*)\epsilon_{\tilde{n},j}.$$} 
  Given such a $j,$ for each $k = m-j,\dots, 1,$ $\iota_{m,n}^*$ sends $$Q_{m-j, k}(\epsilon_{n, j+k+(i+1)})(\epsilon_{n,n},\dots, \epsilon_{n, j+k+i+2})\epsilon_{n,n-j}\mapsto Q_{m-j,k}(\epsilon_{m,j+k})(\epsilon_{m,m},\dots, \epsilon_{m,j+k+1})\epsilon_{m,m-j}.$$ Therefore, $\ker \iota_{i+1,n}^*$ is generated by
  
$$\left\{\begin{array}{c}
    \ker s_{m,n}^*;  \text{ for each } j= m, \dots, 1:  Q_{m, j}(\epsilon_{n,n},\dots, \epsilon_{n, j+i+2}, \epsilon_{n, j+(i+1)});  \ker(s_{(m-j)n}^*)\cdot \epsilon_{n,j+(i+1);}\\
    \text{for each }k = m-j,\dots, 1:  Q_{m-j, k}(\epsilon_{n,n},\dots, \epsilon_{n, j+k+i+2}, \epsilon_{n, j+k+(i+1)}) \cdot \epsilon_{n,n-j}
  \end{array} \right\}.$$
  
  The new relations in $\mathrm{CH}^\star(\Hilb^n(C|p)_{\leq i})$ introduced by the weighted blow-up are hence generated by 
  $$\left\{\begin{array}{c}
    Q_{n,i+1}(\epsilon_{n,n},\dots, \epsilon_{n,i+2}, \epsilon_{n,i+1}); \ \ker s_{m,n}^* \cdot \epsilon_{n,i+1};
    \\
     \text{for each }j=m, \dots, 1:  Q_{m, j}(\epsilon_{n,n},\dots, \epsilon_{n, j+i+2}, \epsilon_{n, j+(i+1)})\cdot\epsilon_{n,i+1}; \  \ker(s_{(m-j)n}^*)\epsilon_{n,j+(i+1)}\cdot \epsilon_{n,i+1};\\
    \text{for each }k = m-j,\dots, 1:  Q_{m-j, k}(\epsilon_{n,n},\dots, \epsilon_{n, j+k+i+2}, \epsilon_{n, j+k+(i+1)})\cdot \epsilon_{n,n-j} \cdot \epsilon_{n,i+1}
  \end{array} \right\}$$
  
  We notice that among the above terms,  $Q_{m-j, k}(\epsilon_{n,n},\dots, \epsilon_{n, j+k+i+2}, \epsilon_{n, j+k+(i+1)}) \cdot \epsilon_{n,n-j} \cdot \epsilon_{n,i+1}$ lies in the ideal generated by $Q_{m-j, k}(\epsilon_{n,n},\dots, \epsilon_{n, j+k+i+2}, \epsilon_{n, j+k+(i+1)}) \cdot \epsilon_{n,n-j},$ which is already part of relations in $\mathrm{CH}^\star(\Hilb^n(C|p)_{\leq i+1})$ by inductive hypothesis.
  
  Similarly, $\ker s^*_{(n-j),n}\cdot \epsilon_{n,j+(i+1)}$ is a relation in $\mathrm{CH}^\star(\Hilb^n(C|p)_{\leq i+1}).$ The rest of the relations are precisely the relations corresponding to $j = i+1$ (and $k = m,\dots, 1$) as desired.
\end{proof}

In the proof, we have used the following elementary fact about rings and ideals.
\begin{lemma}\label{lem:rings}
  Let $\tilde{\phi}: R\to S$ be a surjection of rings, and let $I\subset S$ be an ideal such that $\langle b_1,\dots b_n\rangle = I.$ Pick any $a_i\in R$ such that $\tilde{\phi}(a_i) = b_i,$ then $\tilde{\phi}^{-1}(I) = \langle \ker \tilde{\phi}, a_1,\dots,a_n\rangle.$
\end{lemma}
\begin{proof}
  Firstly, $\tilde{\phi}^{-1}(I)$ is an ideal, so $b_i\in \tilde{\phi}^{-1}(I), \ker \tilde{\phi}\subset \tilde{\phi}^{-1}(I)$ implies that $\langle \ker \tilde{\phi}, a_1,\dots,a_n\rangle\subset \tilde{\phi}^{-1}(I).$

  Suppose $\ell \in \tilde{\phi}^{-1}(I),$ then write $\tilde{\phi}(\ell) = s_1 b_1 + \dots + s_nb_n.$ Pick any $r_i\in R$ such that $\tilde{\phi}(r_i) = s_i,$ then $$\tilde{\phi}(\sum_{i=1}^n r_i a_i) = \sum_{i=1}^n s_i b_i = \tilde{\phi}(\ell),$$ so $(\sum_{i=1}^n r_i a_i)-\ell\in \ker \tilde{\phi},$ therefore $\ell\in\langle  \ker \tilde{\phi},a_1,\dots, a_n\rangle.$
\end{proof}

\subsection{The Chow ring of $\textbf{\text{Hilb}}^{\boldsymbol{n}}{\boldsymbol{(C|D)}}_{\le \bi}$}

When the divisor $D$ has more than one point, iterating the above calculations gives a presentation of $\mathrm{CH}^\star(\Hilb^n(C|D)_{\leq \bi}).$

Let us start with adapting the Definitions \ref{defn:Qmh} and \ref{defn:sij} to when there are more than one marking present.

\begin{definition}\label{defn:sij}
    For $p_r \in D$, $j\geq i,$ define $$s_{i,j}^{(p_r)}: \Sym^i(C) \ra \Sym^j(C)$$ as the map given by $Z\mapsto Z + (j-i)p_r.$
\end{definition}

The point $p_r$ specifies a hypersurface $H_{n,n}^{(p_r)}\subset \Sym^n(C)$ defined by one point lying on $p_r$ and a corresponding line bundle $\cL_{p_r}$ which behaves as specified by Lemma \ref{lem:LonHnn}. Then we can generalize definition \ref{defn:Qmh} as follows.
 
 \begin{definition} We define the polynomial in $\CHs(\Sym^n(C))[t_m,\dots,t_h]$
 $$Q_{m,h}^{(r)}(t_{m},\dots,t_{h+1},t_h) :=\prod_{k=1}^h (kt_h + c_1(\mathcal{L}_{p_r})-\sum_{j=1}^{m-h}(k+j)t_{j+h})$$ for $r = 1,\dots, \ell$.
\end{definition}

\begin{definition}
    For $j\ge i_r,$ denote $\epsilon_j^{(r)}=[\cE_j^{(r)}]$ as the classes of the exceptional divisors of the weighted blow-ups of Theorem \ref{thm:main} needed to get to $\Hilb^n(C|D)_{\leq \bi}$. They correspond to the {closures of} loci parameterizing length-$n$ subschemes where exactly $i$ points are supported on the last bubble along $p_r$.
\end{definition}

\begin{remark}
    Let $\bi$ be a vector with $r$-th entry equal to $h+1,$ and let $\cZ^{(r)}_{h+1}\subset \Hilb^n(C|D)_{\leq \bi}$ be the closed substack of length-$n$ subschemes where the endpoint along $p_r$ has support $h+1.$ Then we have $c^{\Gm}_{\mathrm{top}}(\cN_{\cZ^{(r)}_{h+1}} \Hilb^n(C|D)_{\leq \bi}) = Q_{n,h}^{(r)}(\epsilon^{(r)}_n, \dots, \epsilon^{(r)}_{h+1}, t).$
\end{remark}

\begin{customthm}{D}\label{thm:chowringD}
    The Chow ring $\mathrm{CH}^\star(\Hilb^n(C|D)_{\leq \bi})$ has a presentation 
  
  $$\mathrm{CH}^\star(\Hilb^n(C|D)_{\leq \bi}) = \frac{\mathrm{CH}^\star(\mathrm{Sym}^n(C))[\epsilon^{(r)}_{j}\mid i_r+1\leq j\leq n, r = 1,\dots, \ell]}{\left\{\begin{array}{c}
   \text{for each }j = n, \dots, i_r+1:\  Q_{n,j}^{(r)}(\epsilon^{(r)}_{n},\dots,\epsilon^{(r)}_{j}); \\ \ker({(s^{(p_r)}_{(n-j), n})}^*) \cdot \epsilon^{(r)}_{j};\\ \text{for each }k = n-j, \dots, 1: \ Q_{n-j,k}^{(r)}(\epsilon^{(r)}_{n},\dots, \epsilon^{(r)}_{j+k})\cdot\epsilon^{(r)}_{n-j} 
  \end{array}\right\}_{r = 1,\dots, \ell}}.$$
\end{customthm}

\begin{remark}
    The formula does not depend on the order of the blow-ups, because for $j\neq k,$ the blow-up centers $\cZ_{\boldsymbol{n}, \boldsymbol{j}-1_j}$ and $\cZ_{\boldsymbol{n}, \boldsymbol{j}-1_k}$ intersect transversely. Therefore, when either of the centers is blown up, the total transform of the other agrees with its proper transform.
\end{remark}


\section{The boundary strata of $\Hilb^{{n}}{{(C|D)}}$}\label{subsec:strat}
The exceptional divisors $\epsilon_{j}^{(r)}$ intersect transversely and give $\Hilb^{n}(C|D)$ a divisorial logarithmic structure. The boundary strata of $\Hilb^{n}(C|D)$ are indexed by the lengths of support on each component. We record this by tuples $(m,\underline{\boldsymbol{\nu}})  = (m, \boldsymbol{\nu}_1, ,\dots ,\boldsymbol{\nu}_\ell),$ where \begin{enumerate}
    \item $m\in \mathbb{Z}_{\geq 0}$ denotes the length of the subscheme supported on the interior of the curve $C\setminus D,$
    \item for $1\leq i\leq \ell,$ $\boldsymbol{\nu}_i$ is a (possibly empty) sequence of positive integers that records the lengths of subschemes supported on irreducible components from the bubble adjacent to the interior (along $p_i$) to the one containing the endpoint $\tilde{p}_i$. We denote the length of $\boldsymbol{\nu}_i$ as $k_i.$ 
\end{enumerate}

Such a tuple $(m,\underline{\boldsymbol{\nu}})$ specifies a locally closed substack $\Hilb^{n}(C|D)_{\leq \bi}^{(m,\underline{\boldsymbol{\nu}})}\subseteq \Hilb^{n}(C|D)$ parametrizing subschemes on expansions with the prescribed lengths of support at components. See Figure \ref{fig:partition} for an example.

\begin{figure}[htp]
\resizebox{.2\textwidth}{!}{
\begin{tikzpicture}[x=0.75pt,y=0.75pt,yscale=-1,xscale=1]

\draw [color={rgb, 255:red, 74; green, 144; blue, 226 }  ,draw opacity=1 ]   (87.04,12.97) -- (65.3,71.28) ;
\draw [color={rgb, 255:red, 74; green, 144; blue, 226 }  ,draw opacity=1 ]   (27.73,10) -- (5.99,68.32) ;
\draw [color={rgb, 255:red, 0; green, 0; blue, 0 }  ,draw opacity=1 ]   (9.94,23.84) -- (91.98,23.84) ;
\draw  [fill={rgb, 255:red, 0; green, 0; blue, 0 }  ,fill opacity=1 ] (84.08,24.33) .. controls (84.08,23.51) and (83.41,22.85) .. (82.59,22.85) .. controls (81.78,22.85) and (81.11,23.51) .. (81.11,24.33) .. controls (81.11,25.15) and (81.78,25.82) .. (82.59,25.82) .. controls (83.41,25.82) and (84.08,25.15) .. (84.08,24.33) -- cycle ;
\draw  [fill={rgb, 255:red, 0; green, 0; blue, 0 }  ,fill opacity=1 ] (23.78,24.33) .. controls (23.78,23.51) and (23.12,22.85) .. (22.3,22.85) .. controls (21.48,22.85) and (20.82,23.51) .. (20.82,24.33) .. controls (20.82,25.15) and (21.48,25.82) .. (22.3,25.82) .. controls (23.12,25.82) and (23.78,25.15) .. (23.78,24.33) -- cycle ;
\draw  [color={rgb, 255:red, 208; green, 2; blue, 27 }  ,draw opacity=1 ][fill={rgb, 255:red, 208; green, 2; blue, 27 }  ,fill opacity=1 ] (19.8,36.22) -- (20.59,37.01) -- (17.65,39.95) -- (20.59,42.89) -- (19.8,43.68) -- (16.86,40.74) -- (13.92,43.68) -- (13.13,42.89) -- (16.07,39.95) -- (13.13,37.01) -- (13.92,36.22) -- (16.86,39.16) -- cycle ;
\draw [color={rgb, 255:red, 74; green, 144; blue, 226 }  ,draw opacity=1 ]   (5,42.62) -- (27.73,94.02) ;
\draw  [color={rgb, 255:red, 208; green, 2; blue, 27 }  ,draw opacity=1 ][fill={rgb, 255:red, 208; green, 2; blue, 27 }  ,fill opacity=1 ] (23.24,73.26) -- (24.03,74.05) -- (21.09,76.99) -- (24.03,79.93) -- (23.24,80.72) -- (20.3,77.78) -- (17.36,80.72) -- (16.56,79.93) -- (19.5,76.99) -- (16.56,74.05) -- (17.36,73.26) -- (20.3,76.2) -- cycle ;
\draw  [color={rgb, 255:red, 208; green, 2; blue, 27 }  ,draw opacity=1 ][fill={rgb, 255:red, 208; green, 2; blue, 27 }  ,fill opacity=1 ] (17.36,60.86) -- (18.15,61.65) -- (15.21,64.59) -- (18.15,67.53) -- (17.36,68.32) -- (14.42,65.38) -- (11.48,68.32) -- (10.68,67.53) -- (13.62,64.59) -- (10.68,61.65) -- (11.48,60.86) -- (14.42,63.8) -- cycle ;
\draw  [color={rgb, 255:red, 208; green, 2; blue, 27 }  ,draw opacity=1 ][fill={rgb, 255:red, 208; green, 2; blue, 27 }  ,fill opacity=1 ] (74.69,49.98) -- (75.48,50.78) -- (72.54,53.71) -- (75.48,56.65) -- (74.69,57.45) -- (71.75,54.51) -- (68.81,57.45) -- (68.02,56.65) -- (70.96,53.71) -- (68.02,50.78) -- (68.81,49.98) -- (71.75,52.92) -- cycle ;
\draw  [fill={rgb, 255:red, 0; green, 0; blue, 0 }  ,fill opacity=1 ] (50.96,23.84) .. controls (50.96,23.02) and (50.3,22.36) .. (49.48,22.36) .. controls (48.66,22.36) and (48,23.02) .. (48,23.84) .. controls (48,24.66) and (48.66,25.32) .. (49.48,25.32) .. controls (50.3,25.32) and (50.96,24.66) .. (50.96,23.84) -- cycle ;
\draw  [color={rgb, 255:red, 208; green, 2; blue, 27 }  ,draw opacity=1 ][fill={rgb, 255:red, 208; green, 2; blue, 27 }  ,fill opacity=1 ] (67.25,20.25) -- (68.05,21.05) -- (65.11,23.99) -- (68.05,26.93) -- (67.25,27.72) -- (64.31,24.78) -- (61.37,27.72) -- (60.58,26.93) -- (63.52,23.99) -- (60.58,21.05) -- (61.37,20.25) -- (64.31,23.19) -- cycle ;

\draw (13,11.4) node [anchor=north west][inner sep=0.75pt]  [font=\tiny]  {$p_{1}$};
\draw (43,11.4) node [anchor=north west][inner sep=0.75pt]  [font=\tiny]  {$p_{2}$};
\draw (88,11.4) node [anchor=north west][inner sep=0.75pt]  [font=\tiny]  {$p_{3}$};

\end{tikzpicture}

}
\caption{A point in $\Hilb^5(C|D)^{1, ((1,2), \varnothing, (1))}.$ Here, $D = p_1 + p_2 + p_3$ and the partitions are ordered accordingly.}
\label{fig:partition}
\end{figure}

To describe the substack, we recall that the $\mathbb{C}^\star$-action on $\Sym^i(\mathbb{A}^1)\cong \mathbb{A}^i$ induced by the $\mathbb{C}^\star$-action on $\mathbb{A}^1$ has weights $(1,\dots, i),$ so that $[(\Sym^i(\mathbb{A}^1)\setminus 0)/\mathbb{C}^\star]\cong \mathcal{P}(1,\dots, i),$ and the coordinates $[e_1:\dots: e_i]$ on $\mathcal{P}(1,\dots, i)$ correspond to the elementary symmetric polynomials on $i$ variables. Restricting the $\mathbb{C}^\star$-action to $\mathbb{C}^\star$ and $\Sym^i(\mathbb{C}^\star),$ the quotient stack $[\Sym^i(\mathbb{C}^\star)/\mathbb{C}^\star]$ is isomorphic to $\mathcal{P}(1,\dots, i)\setminus \mathbb{V}(e_i).$


We have an isomorphism \[\Hilb^{n}(C|D)^{(m,\underline{\boldsymbol{\nu}})}\cong \Sym^{m}(C\setminus D)\times \prod_{i=1}^{\ell}\prod_{j = 1}^{k_i}[\mathrm{Sym}^{\nu_i^{(j)}}(\mathbb{C}^\star)/\mathbb{C}^\star].\]

The product description allows us to classify the stabilizer groups at closed points in $\Hilb^{n}(C|D).$ We recall that the stabilizer of a point $[a_1:\dots:a_i]\in \mathcal{P}(1,\dots, i)$ is cyclic group of order $\mathrm{gcd}(k\in \{1,\dots, i\}: a_k\neq 0)$ embedded as roots of unity in $\mathbb{C}^\star$. Thus, the stabilizer group at each point is a product of cyclic groups, and the cyclic group factor coming from the bubble specified by $(i,j)\in \{1,\dots, \ell\}\times \{1,\dots, k_i-1\}$ must have order dividing $\nu_{i}^{(j)}.$

We also relate the locally closed substacks $\Hilb^{n}(C|D)^{(m,\underline{\boldsymbol{\nu}})}$ to the exceptional divisors introduced by weighted blow-ups. For $1\leq j\leq k_i,$ let $N_{i,j}:=\sum_{k=k_i-j+1}^{k_i}\nu_{i}^{(k)},$ so that $N_{i,j}$ is the total length of subscheme supported on the first $j$ bubbles from the endpoint. The closure of $\Hilb^{n}(C|D)^{(m,\underline{\boldsymbol{\nu}})}$ in $\Hilb^n(C|D)$ is a smooth, regularly embedded closed substack of codimension $\sum_{i=1}^\ell k_i,$ and its cycle class in $A^{\star}(\Hilb^n(C|D))$ has class $\prod_{i=1}^{\ell}\prod_{j=1}^{k_i}\epsilon^{(i)}_{n, N_{i,j}}.$


    

\section{Generating function of Euler characteristics}\label{sec:chi}
We determine the Euler characteristics of $\Hilb^n(C|D)$ by calculating their classes in the Grothendieck ring of varieties. We use the locally closed stratification on $\mathrm{Hilb}^n(C|D)$ itself, which is along the lines of §\ref{subsec:strat} and complementary to the weighted blow-up perspective that we have been taking previously.

Let $Z_C(t)= \sum_{m\geq 0} \mathrm{Sym}^m(C) t^m\in K_0(\mathsf{Var})[\![t]\!]$ be the motivic Zeta function of $C.$ We determine classes of $\Hilb^n(C|D)$ in terms of $Z_C(t).$

\begin{customthm}{E}\label{thm:chi}
     We have $$\sum_{n\geq 0}\Hilb^n(C|D) t^n = Z_C(t)\left(\frac{(1-\mathbb{L}t)(1-t)}{1-(\mathbb{L}+1)t}\right)^\ell.$$
\end{customthm}

\begin{proof}
    From the previous section, $\Hilb^n(C|D)$ is stratified by \[\Hilb^{n}(C|D)^{(m,\underline{\boldsymbol{\nu}})}\cong \Sym^{m}(C\setminus D)\times \prod_{i=1}^{\ell}\prod_{j = 1}^{k_i}[\mathrm{Sym}^{\nu_i^{(j)}}(\mathbb{C}^\star)/\mathbb{C}^\star]\] for all tuples $(m,\underline{\boldsymbol{\nu}})$ where $0\leq m \leq n,$ and $m+\sum_{i=1}^\ell |\boldsymbol{\nu}_i|=n.$ Therefore, \begin{align*}\sum_{n\geq 0}\Hilb^n(C|D) t^n & = \sum_{n\geq 0}\sum_{\substack{(m, \underline{\boldsymbol{\nu}}):\\ m+\sum_{i=1}^\ell |\boldsymbol{\nu}_i|=n}}[\Hilb^{n}(C|D)^{(m,\underline{\boldsymbol{\nu}})}]\cdot t^{n}\\
    & = \sum_{(m, \underline{\boldsymbol{\nu}})}(\Sym^{m}(C\setminus D)\cdot t^m)\cdot \prod_{i=1}^{\ell}\left(\sum_{\underline{\boldsymbol{\nu}}}\prod_{j = 1}^{k_i}[\mathrm{Sym}^{\nu_i^{(j)}}(\mathbb{C}^\star)/\mathbb{C}^\star]\cdot t^{\nu_i^{(j)}}\right)\\ & = \left(\sum_{m\geq 0}\Sym^{m}(C\setminus D)\cdot t^m\right)\left(\sum_{k = 0}\left(\sum_{j=1}^\infty [\Sym^j \mathbb{C}^\star/\mathbb{C}^\star]\cdot t^j\right)^k\right)^\ell\\ & = \left(\sum_{m\geq 0}\Sym^{m}(C\setminus D)\cdot t^m\right)\frac{1}{\left(1-\sum_{j = 1}^{\infty} [\Sym^j \mathbb{C}^\star/\mathbb{C}^\star]\cdot t^j \right)^\ell}\end{align*}

    Since $Z_{(\text{-})}(t)$ is motivic, we have $Z_{C}(t) = Z_{C\setminus D}(t)\left(Z_{\mathrm{pt}}(t)\right)^\ell = Z_{C\setminus D}(t)\cdot (1-t)^{-\ell}.$ Also, the quotient stack $[\mathrm{Sym}^{\ell}(\mathbb{C}^\star)/\mathbb{C}^\star]$ has $\mathbb{A}^{\ell-1}$ as coarse moduli space, so $[\mathrm{Sym}^{j}(\mathbb{C}^\star)/\mathbb{C}^\star] = [\mathbb{L}^{j-1}]\in K_0(\mathsf{Var}).$ Using these two facts, the generating function simplifies to $${(1-t)^{\ell}\cdot Z_C(t)}\cdot {\left(\frac{1-(\mathbb{L}+1)t}{1-\mathbb{L}t}\right)^{-\ell}}$$ as desired.
\end{proof}

When $C = \mathbb{P}^1,$ the motivic zeta function is $$\sum_{m\geq 0}\mathbb{P}^m\cdot t^m = \frac{1}{(1-t)(1-\mathbb{L}t)},$$ allowing us to completely determine the generating functions in this case. 
\begin{corollary}
  We have $$\sum_{n\geq 0}\Hilb^n(\mathbb{P}^1|0+\infty)t^n = \frac{1-t}{1-\mathbb{L}t}\frac{1}{\left(1-\frac{t}{1-\mathbb{L}t}\right)^2} = \frac{(1-t)(1-\mathbb{L}t)}{(1- (\mathbb{L}+1) t)^2},$$ $$\sum_{n\geq 0}\Hilb^n(\mathbb{P}^1|0)t^n  = \frac{1}{1-(\mathbb{L}+1)t}.$$
\end{corollary}

\begin{remark}
  In other words, $[\Hilb^n(\mathbb{P}^1|0)] = [(\mathbb{P}^1)^n]$ in the Grothendieck ring of varieties. This is because as toric stacks, they have the same number of cones in each dimension.
\end{remark}

We recall that by Macdonald's formula \cite{Macdonald_1962}, the motivic zeta function $Z_C(t)$ specializes to

\begin{tabular}{|c|c|c|c|}
  \hline 
  Invariant & $\Sym^n(C)$ & $\Hilb^n(C|D)$ & Ring\\
  \hline Serre characteristics
   & $\displaystyle \frac{(1-t)^{[H^1(C)]}}{(1-t)(1-\mathbb{L}t)}$ & $ \displaystyle \frac{(1-t)^{[H^1(C)]}}{(1-t)(1-\mathbb{L} t)}\left(\frac{(1-\mathbb{L}t)(1-t)}{1-(\mathbb{L}+1)t}\right)^\ell$ & $K_0(\mathsf{MHS})[\![t]\!]$\\
  \hline  \begin{tabular}{c}
       Hodge--Deligne  \\
        polynomials
  \end{tabular}   & $\displaystyle \frac{(1-ut)^g(1-vt)^g}{(1-t)(1-uvt)}$ & $\displaystyle \frac{(1-ut)^g(1-vt)^g}{(1-t)(1-uvt)}\left(\frac{(1-uvt)(1-t)}{1-(uv+1)t}\right)^\ell$ & $\mathbb{Z}[u,v][\![t]\!]$\\
  \hline \begin{tabular}{c}
       Poincaré \\
        polynomials
  \end{tabular} & $\displaystyle \frac{(1-xt)^{2g}}{(1-t)(1-x^2t)}$ & $\displaystyle \frac{(1-xt)^{2g}}{(1-t)(1-x^2t)}\left(\frac{(1-x^2t)(1-t)}{1-(x^2+1)t}\right)^\ell$ & $\mathbb{Z}[x][\![t]\!]$\\
  \hline Euler characteristics & $\displaystyle (1-t)^{2g-2}$ & $\displaystyle (1-t)^{2g-2} \left(\frac{(1-t)^2}{1-2t}\right)^{\ell}$ & $\mathbb{Z}[\![t]\!]$\\
  \hline
\end{tabular}

\appendix
\section{Toric structure of $\Hilb^n(\PP^1|0+\infty)$ via Chow quotients}
\begin{center}
    by Dhruv Ranganathan
\end{center}

Corollary \ref{cor:P1} of this paper states that that the logarithmic Hilbert scheme $\Hilb^n(\PP^1|0+\infty)$ is toric; indeed, it is presented as a sequence of toric blowups of a toric variety. A closely related variant of this space is a well-known toric variety: the toric variety associated with the secondary polytope of the length-$n$ interval. See results of Kennedy-Hunt \cite{kh21} and also \cite[§6]{MR20}. Using this, one can check that the logarithmic Hilbert scheme $\Hilb^n(\mathbb{P}^1|0+\infty)$ is toric. Since $\Hilb^n(\PP^1)$ is $\PP^n$, which is also toric, one might guess the intermediate case $\Hilb^n(\mathbb{P}^1|0)$. We briefly sketch this and leave many details to the reader.

Consider the ``rubber" logarithmic Hilbert scheme $\Hilb^n_{\mathrm{rub}}(\mathbb{P}^1|0+\infty)$, parameterizing cycles on expansions of $\mathbb{P}^1$ along $0$ and $\infty$, but where two points are identified if they differ by $\mathbb{C}^*$-scaling in the main component. The moduli problem is represented by a Deligne--Mumford stack with logarithmic structure, by the same methods as \cite{MR20}.

Separately, the linear system $\mathbb{P}^n$ of degree-$n$ divisors on $\mathbb{P}^1$. There is a $\mathbb{C}^*$-scaling action on $\mathbb{P}^n$ induced by the standard $\mathbb{C}^*$-action on $\PP^1$. Let $H$ denote the (normalized) Chow quotient of $\mathbb{P}^n$ by this action. The Chow quotient parameterizes isomorphism classes of $\mathbb{C}^*$-equivariant degenerations of $\mathbb{P}^1$ with a degree-$n$ divisor. This follows from, for instance, the ``finite case" of Alexeev's work on complete semi-abelian moduli \cite[§2.10-12]{Alexeev2002}, or by work of Ascher--Molcho \cite{AM16}. As such, there is a moduli map:
	\[h:\Hilb^n_{\mathrm{rub}}(\mathbb{P}^1|0+\infty) \to H\]
\begin{proposition} The morphism $h$ is an isomorphism on coarse moduli spaces. \end{proposition}

\begin{proof} By the moduli interpretation, one easily checks that $h$ is finite. It is also clearly birational, since the interior parameterizes $n$ points on $\mathbb{C}^*$. The Chow quotient is normal by construction, so the map is an isomorphism by Zariski's main theorem. \end{proof}

As mentioned above, the Chow quotient of the linear system is the toric variety associated with the secondary polytope of the length-$n$ interval \cite{KSZ91}. Therefore,

\begin{corollary} The coarse moduli space of $\Hilb^n_{\mathrm{rub}}(\mathbb{P}^1|0+\infty)$ is a toric variety. \end{corollary}

One can adapt this argument to the non-rubber moduli space using the approach of Kennedy-Hunt \cite{kh21}; we do not pursue this here. Nevertheless, inspecting the map \[\Hilb^n(\mathbb{P}^1|0+\infty) \to \Hilb^n_{\mathrm{rub}}(\mathbb{P}^1|0+\infty),\]it is straightforward see that it is a flat family of chains $\mathbb{P}^1$'s, which provides a plausibility argument for the toric-ness of $\Hilb^n(\mathbb{P}^1|0+\infty)$. 

One note of caution is that while $\Hilb^n_{\mathrm{rub}}(\PP1|0+\infty)$ and $H$ are isomorphic, the logarithmic boundary divisor on the former is a strict subset of the toric boundary divisor boundary on the latter. In addition to this, a drawback of this argument is that it seems not to apply in the case of $(\PP^1|0)$, and certainly not to other logarithmic curves.

\bibliographystyle{alpha}
\bibliography{lhwbu.bib}
\,

\end{document}